\documentclass{amsart}

\usepackage{tikz-cd}
\usepackage[hyphens]{url}
\usepackage{url}
\usepackage[page]{appendix}
\usepackage{amsthm,amsmath,amssymb,amsfonts,dsfont}
\usepackage{amsxtra}
\usepackage[colorlinks, linkcolor=blue, citecolor=blue]{hyperref}
\usepackage[indentafter]{titlesec}
\titleformat{name=\section}{}{\thetitle.}{0.8em}{\centering\scshape}
\usepackage{mathrsfs}
\usepackage{bm}
\usepackage{titlesec}
\numberwithin{equation}{section}
\titleformat{\subsection}[runin]
  {\normalfont\bfseries}{\thesubsection}{1em}{}
\newtheorem{definition}{\textbf{Definition}}[section]
\newtheorem{theorem}[definition]{\textbf{Theorem}}
\newtheorem{corollary}[definition]{\textbf{Corollary}}
\newtheorem{lemma}[definition]{\textbf{Lemma}}
\newtheorem{proposition}[definition]{\textbf{Proposition}}
\newtheorem{example}[definition]{\textbf{Example}}
\newtheorem{remark}[definition]{\textbf{Remark}}
\newtheorem{letterthm}{Theorem}

\def\d{\mathrm{d}}
\def\car{\curvearrowright}
\def\supp{\mathrm{supp \, }}
\def\prob{\mathrm{Prob}}
\def\BL{B(L^2(M))}
\def\CS{\mathcal{S}}
\def\CB{\mathcal{B}}
\def\CP{\mathcal{P}}
\def\CA{\mathcal{A}}
\def\har{\mathrm{Har}}
\def\id {\mathrm{id}}
\def\lag {\langle}
\def\rag {\rangle}

\def\N {\mathbb{N}}
\def\R {\mathbb{R}}
\def\C {\mathbb{C}}
\def\UCP {\mathrm{UCP}}
\def\FB {\partial_F \Gamma }
\def\CFB {C(\partial_F \Gamma)}
\def\rad {\mathrm{Rad}}
\def\SA {\mathrm{SA}}
\def\SAAM{\mathrm{SA}_{\mathrm{am}}}
\def\AD{\mathrm{Ad}}
\def\TN{\mathscr{N}}
\def\CR{\mathcal{R}}
\def\SR{\mathrm{SR}}

\title[Noncommutative topological boundaries and amenable IRIAs]{Noncommutative topological boundaries and amenable invariant random intermediate subalgebras}
\author[Shuoxing Zhou]{Shuoxing Zhou\\With an appendix by Tattwamasi Amrutam and Yongle Jiang}
\address{\'Ecole Normale Sup\'erieure\\ D\'epartement de math\'ematiques et applications\\ 45 rue d'Ulm\\ 75230 Paris Cedex 05\\ FRANCE}
\email{shuoxing.zhou@ens.psl.eu}

\address{Institute of Mathematics of the Polish Academy of Sciences, ul.~\'Sniadeckich 8, 00--656 Warszawa, Poland}
\email{tattwamasiamrutam@impan.pl}

\address{School of Mathematical Sciences, Dalian University of Technology, Dalian, 116024, China}
\email{yonglejiang@dlut.edu.cn}

\begin{document}

\begin{abstract}
As an analogue of the topological boundary of discrete groups $\Gamma$, we define the noncommutative topological boundary of tracial von Neumann algebras $(M,\tau)$ and apply it to generalize the main results of \cite{AHO23}, showing that for a trace preserving action  $\Gamma\car(A,\tau_A)$ on an amenable tracial von Neumann algebra, any $\Gamma$-invariant amenable intermediate subalgebra between $A$ and $\Gamma\ltimes A$ is necessarily a subalgebra of $\rad(\Gamma)\ltimes A$. By taking $(A,\tau_A)=L^\infty(X,\nu_X)$ for a free pmp action $\Gamma\car(X,\nu_X)$, we obtain a similar result for the invariant subequivalence relations of $\CR_{\Gamma\car X}$.
\end{abstract}

\maketitle

\section{Introduction}
In the study of countable discrete groups $\Gamma$, the topological $\Gamma$-space is a very useful tool. Topological boundaries are particularly important among all kinds of topological $\Gamma$-spaces; for example, they have played a central role in the study of C$^*$-simplicity \cite{KK14, BKKO14}. When it comes to a tracial von Neumann algebra $(M,\tau)$, the noncommutative analogue of topological actions $\Gamma\car X$ is the C$^*$-inclusions $M\subset\CA$, i.e., $M$ can be embedded into the C$^*$-algebra $\CA$ as a C$^*$-subalgebra. In this paper, we extend several important topological concepts from group dynamical systems to the noncommutative setting and present some of their applications.

The study of maximal amenable subalgebras (e.g. \cite{Pop83}) and invariant intermediate subalgebras under group actions (e.g. \cite{Suz18, BH22, AH24}) is an important topic in von Neumann algebras. In this paper, we explore the intersection of these two areas by applying noncommutative topological boundaries to investigate maximal amenable invariant intermediate subalgebras. This leads to the following main theorem.
\begin{letterthm}[Theorem \ref{max amenable}]\label{Thm B}
 Let $\Gamma$ be a countable discrete group, $(A,\tau_A)$ be an amenable separable tracial von Neumann algebra, and $\Gamma\car(A,\tau_A)$ be a trace preserving action by $\ast$-isomorphisms. Assume that $N\subset\Gamma\ltimes A $ is an amenable $\Gamma$-invariant von Neumann subalgebra with $A\subset N\subset \Gamma\ltimes A$. Then we must have 
$$N\subset \rad(\Gamma)\ltimes A .$$
In particular, when $\rad(\Gamma)=\{e\}$, $A$ is a maximal invariant amenable subalgebra.
\end{letterthm}

Denote by $\SA(\Gamma\ltimes A )$ the space of von Neumann subalgebras of $\Gamma\ltimes A$ with the Effros-Mar\'{e}chal topology. Then $\SA(\Gamma\ltimes A )$ admits a natural continuous $\Gamma$-action. We define an \textbf{amenable $\Gamma$-invariant random intermediate subalgebra (IRIA)} to be a $\Gamma$-invariant Borel probability measure $\mu\in \prob(\SA(\Gamma\ltimes A ))$ such that $\mu$-a.e. $N\in\SA( \Gamma\ltimes A )$ is amenable with $A\subset N\subset \Gamma\ltimes A$. We also have the following theorem as a random version of Theorem \ref{Thm B}.
\begin{letterthm}[Theorem \ref{amenable IRS for crossed products}]\label{Thm C}
With the same conditions in Theorem \ref{Thm B}, let $\mu\in\prob (\SA( {\Gamma\ltimes A} ))$ be an amenable $\Gamma$-IRIA. Then for $\mu$-a.e. $N\in\SA( \Gamma\ltimes A )$, we have
$$N\subset \rad(\Gamma)\ltimes A .$$
In particular, when $\rad(\Gamma)=\{e\}$, we must have $\mu=\delta_A$.
\end{letterthm}
Theorem \ref{Thm B} and \ref{Thm C} generalize the main results of \cite{AHO23}, where the case that $A=\C$ is proved. Namely, with the tool of noncommutative topological boundaries (the following Theorem \ref{Thm A}), we can provide a new proof that holds for general amenable $A$.

Given a separable tracial von Neumann algebra $(M,\tau)$, we define a \textbf{noncommutative topological $M$-boundary} (Section \ref{Section NC topological boundary}) to be a C$^*$-algebra $\CA$ with an \textbf{essential C$^*$-inclusion $M\subset \CA$} in the sense of \cite{Ham79} (see also Subsection \ref{def ess}). Similar approaches using essential extensions to define boundary-type objects have also appeared in other contexts, such as \cite{KK14, BK21, Bor19, KKSV22}.

Let $X$ be a topological $\Gamma$-boundary and take $x\in X$. Through the Poisson transform $P_{\delta_x}: f\in C(X)\mapsto(g\mapsto f(gx))\in \ell^\infty(\Gamma)$, $C(X)$ can be viewed as a C$^*$-subalgebra of $\ell^\infty(\Gamma)$. Then we have the following theorem as an example of the noncommutative topological boundary.
\begin{letterthm}[Theorem \ref{Gamma A boundary}]\label{Thm A}
 With the same conditions in Theorem \ref{Thm B}, the C$^*$-algebra $\CB_X\subset B(\ell^2(\Gamma)\otimes L^2(A))$ generated by $\Gamma\ltimes A$ and $C(X)\otimes1\subset \ell^\infty(\Gamma)\otimes1$ is a $\Gamma\ltimes A$-boundary.
\end{letterthm}
In particular, by taking $A = \mathbb{C}$ in Theorem \ref{Thm A}, we obtain that the C$^*$-algebra $\mathcal{B}_X = C^*(L(\Gamma), C(X)) \subset B(\ell^2(\Gamma))$ is an $L(\Gamma)$-boundary. Results in the same spirit have appeared in other contexts. For example, in the setting of $\Gamma$-C$^*$-dynamics \cite[Theorem 3.4]{Ham85}, for a $\Gamma$-boundary $X$, the reduced crossed product $\Gamma \ltimes_r C(X)$ is a $\Gamma$-essential extension of $C_r^*(\Gamma)$. Moreover, in noncommutative Poisson boundary theory \cite[Theorem 4.1]{Izu04}, for a probability measure $\mu_\Gamma \in \mathrm{Prob}(\Gamma)$ with the $(\Gamma, \mu_\Gamma)$-Poisson boundary $(B, \nu_B)$, the crossed product $L(\Gamma \curvearrowright (B, \nu_B))$ coincides with the noncommutative Poisson boundary of $L(\Gamma)$.

For a countable nonsingular equivalence relation $\CR$, denote by $\SR(\CR)$ the space of subequivalence relations of $\CR$ (modulo null sets). The study of the space of subequivalence relations was initiated by Alexander Kechris, who showed that it can be endowed with a Polish topology in the pmp case \cite[Theorem 4.13]{Kec17}. Recently, Fran{\c{c}}ois Le Ma{\^\i}tre proved that it still admits a Polish topology in the nonsingular case \cite[Theorem A]{LM24}. In Section \ref{IRR}, we provide another description of this topology for the nonsingular case (Definition \ref{metric}), which is still equivalent to the Polish topology defined in \cite[Section 4.2]{Kec17}, \cite[Discussion before Lemma 3.18]{AFH24}, and \cite[Section 2]{LM24}. 

This Polish topology on $\SR(\CR)$ induces a standard Borel structure. Consider $\CR=\CR_{\Gamma\car X}$ for a free probability measure preserving (pmp) action $\Gamma \car (X,\nu_X)$. Then $\SR(\CR_{\Gamma\car X})$ admits a natural continuous $\Gamma$-action (see (\ref{action on SR})). As an analogue of invariant random subgroup/subalgebra, we define the invariant random subequivalence relation (IRR) on $\CR_{\Gamma\car X}$ (Definition \ref{def IRR}) to be a $\Gamma$-invariant probability measure $\mu\in\prob(\SR(\CR_{\Gamma\car X}))$. As a direct corollary of Theorem \ref{Thm C}, we have the following theorem regarding amenable IRRs.
    \begin{letterthm}[Theorem \ref{Thm IRR}]\label{letterthm IRR}
For a free pmp action $\Gamma \car (X,\nu_X)$, let $\mu\in\prob(\SR(\CR_{\Gamma\car X}))$ be a $\Gamma$-invariant random subequivalence relation (IRR) such that $\mu$-a.e. $\Bar{\CS}\in \SR(\CR_{\Gamma\car X})$ is amenable. Then for $\mu$-a.e. $\Bar{\CS}\in\SR(\CR_{\Gamma\car X})$, we have
$$\Bar{\CS}\subset \Bar{\CR}_{\rad(\Gamma)\car X}.$$
In particular, any amenable $\Gamma$-invariant subequivalence relation of $\CR_{\Gamma\car X}$ must be a subequivalence relation of $\CR_{\rad(\Gamma)\car X}$, up to null sets.
\end{letterthm}

To derive Theorem \ref{letterthm IRR} from Theorem \ref{Thm C}, we use the following natural result regarding the map from subequivalence relations of $\CR$ to intermediate subalgebras of the Cartan inclusion $L^\infty(X)\subset L(\CR)$.
\begin{letterthm}[Theorem \ref{L(S) continuous}]\label{prop D}
Let $\CR\subset X\times X$ be a countable nonsingular equivalence relation on standard probability space $(X,\nu_X)$. Denote by $\SA(L^\infty(X),L(\CR))$ the space of intermediate von Neumann subalgebras of $L^\infty(X)\subset L(\CR)$. Then the map $$L:\Bar{\CS}\in\SR(\CR)\mapsto L(\CS)\in \SA(L^\infty(X),L(\CR))$$ is a homeomorphism.
\end{letterthm}

\subsection*{Organization of the paper.}This paper is divided into eight sections. In addition to the results discussed above, Section \ref{section prox usb} extends the notions of $\mu$-proximality and $\mu$-unique stationary boundaries ($\mu$-USB) to the noncommutative setting and establishes their connection to the recently developed noncommutative Poisson boundary theory \cite{DP20}. In Section \ref{section tight inclusion}, we extend the notion of tight inclusions, along with several related results from \cite{HK24}, to the noncommutative setting. Finally, in Section~\ref{appendix}, generalizations of Theorem~\ref{Thm B} and Theorem~\ref{Thm C} are provided using the Furstenberg boundary. In particular, we show that the assumption that $N$ contains $A$ in the above-mentioned theorems can be removed. We refer the readers to Theorem~\ref{prop: main prop in appendix} and Theorem~\ref{thm:singularmaximal} for precise statements.
\subsection*{New Developments.}In an earlier version of this paper, it was shown that the map $L: \Bar{\CS} \in \SR(\CR) \mapsto L(\CS) \in \SA(L(\CR))$ is continuous for the case that $\CR$ is pmp. Subsequently, Fran{\c{c}}ois Le Ma{\^\i}tre introduced the author to an unpublished but more straightforward proof, developed jointly with his coauthors of \cite{FLMP24} -- Pierre Fima, Issan Patri, and Kunal Mukherjee. Le Ma{\^\i}tre also raised a question: Does this continuity hold when $\CR$ is merely nonsingular? We are now able to answer this question with Theorem \ref{L(S) continuous}, where part of the proof (specifically, formula (\ref{continuity of L-1})) is generalized from the proof by Le Ma{\^\i}tre and his coauthors for the continuity of the pmp case.

This version also contains an Appendix (see Section~\ref{appendix}) written by Tattwamasi Amrutam and Yongle Jiang, which includes an elegant alternative proof of Theorems~\ref{Thm B} and~\ref{Thm C}, with an improvement of the assumptions.
\subsection*{Acknowledgment.}
The author would like to thank his supervisor, Professor Cyril Houdayer, for numerous insightful discussions and valuable comments on this paper. He is also grateful to Professor Hanfeng Li and Professor Yi-Jun Yao for many valuable comments regarding this paper, and to Professor Fran{\c{c}}ois Le Ma{\^\i}tre for clarifying the historical background of the topology on subequivalence relations and for many valuable comments regarding Theorem~\ref{L(S) continuous}. The author would also like to thank Professor Tattwamasi Amrutam and Professor Yongle Jiang for their kind permission to include their alternative proof as an appendix (Section~\ref{appendix}) in this paper.

\section{Preliminaries}\label{Preliminaries}

\subsection{Topological boundaries and amenable radicals of groups}\label{group boundary} Let $\Gamma$ be a countable discrete group. A compact Hausdorff $\Gamma$-space $X$ is a \textbf{$\Gamma$-boundary} if it satisfies one of the following equivalent conditions \cite{Aze70}:
\begin{itemize}
    \item [(1)] For any $\nu\in\prob(X)$, the weak* closure of the $\Gamma$-orbit $\Gamma\nu$ contains all Dirac measures on $X$, i.e., $\delta_X\subset\overline{\Gamma\nu}^{\mathrm{w}^*}$;
    \item [(2)] For any $\nu\in\prob(X)$, then the weak* closed convex hull $\overline{\mathrm{Convex}(\Gamma\nu)}^{\mathrm{w}^*}=\prob(X)$;
    \item [(3)] For any $\nu\in\prob(X)$, the Poisson transform 
    $$P_\nu:f\in C(X)\mapsto\left(g\mapsto \int_X f(gx)\d \nu(x)\right)\in \ell^\infty(\Gamma)$$
    is isometric.
\end{itemize}
In particular, any $\Gamma$-boundary is $\Gamma$-minimal, i.e., for any $x\in X$, $\overline{\Gamma x}=X$.

Following \cite{Fur73}, the \textbf{$\Gamma$-Furstenberg boundary} $\partial_F \Gamma$ is the unique maximal $\Gamma$-boundary in the sense that any $\Gamma$-boundary is a continuous $\Gamma$-equivariant image of $\partial_F \Gamma$. 

The \textbf{amenable radical} $\rad(\Gamma)$ of $\Gamma$ is the unique maximal amenable normal subgroup of $\Gamma$ in the sense that any amenable normal subgroup $\Lambda$ of $\Gamma$ must be a subgroup of $\rad(\Gamma)$. Denote by $\mathrm{Bnd_m}(\Gamma)$ the set of metrizable $\Gamma$-boundaries. Following \cite[Proposition 7]{Fur03}, we have
$$\rad(\Gamma)=\mathrm{Ker}(\Gamma\car\partial_F \Gamma)=\cap_{X\in\mathrm{Bnd_m}(\Gamma)}\mathrm{Ker}(\Gamma\car X).$$

\subsection{Injective envelopes of \texorpdfstring{C$^*$}{}-algebras.}\label{def ess}
A map $\CP:\CB\to\mathcal{C}$ between C$^*$-algebras is called a \textbf{unital completely positive (ucp)} map if it satisfies that $\CP(1)=1$ and for any $n\geq1$, $\CP\otimes I_n:M_n(\CA)\to M_n(\CB)$ is positive. Such a ucp map $\CP$ is further called \textbf{completely isometric} if $\CP\otimes I_n$ is isometric for any $n\geq 1$.

Fix a C$^*$-algebra $\CA$. Following \cite{Ham79}, for a C$^*$-inclusion $\CA\subset\CB$, we say that the inclusion $\CA\subset\CB$ is \textbf{essential} or $\CB$ is an \textbf{essential extension} of $\CA$ if for any other C$^*$-inclusion $\CA\subset \mathcal{C}$, any $\CP\in\UCP_{\CA}(\CB,\mathcal{C})$ is completely isometric. Here $\UCP_{\CA}(\CB,\mathcal{C})$ is the set of \textbf{$\CA$-bimodular} ucp maps from $\CB$ to $\mathcal{C}$, i.e., ucp maps $\CP:\CB\to\mathcal{C}$ with $\CP(xby)=x\CP(b)y$ for any $x,y\in\CA$ and $b\in \CB$.

The \textbf{injective envelope} of $\CA$, denoted by $I(\CA)$, is an essential extension of $\CA$ that is also injective as a C$^*$-algebra, which always exists and is unique up to $\CA$-bimodular $\ast$-isomorphisms. The injective envelope $I(\CA)$ is also the \textbf{universal essential extension} of $\CA$ in the sense that any essential extension of $\CA$ is a subalgebra of $I(\CA)$, up to $\CA$-bimodular $\ast$-isomorphisms.

\subsection{Hyperstates and ucp maps.}\label{hyperstate}
Fix a tracial von Neumann algebra $(M, \tau )$. Following \cite{DP20}, for a  C$^*$-inclusion $M\subset\CA$ and a state $\psi$ on $\CA$, we say that $\psi$ is a $\tau$-\textbf{hyperstate} if $\psi|_M =\tau$. And we say that a $\tau$-hyperstate $\psi$ is a \textbf{hypertrace} if for any $x\in M$, $x\psi=\psi x$ (i.e. $\psi(x\,\cdot\,)=\psi(\,\cdot\,x)$). 

We denote by $\mathcal{S}_\tau(\CA)$ the set of $\tau$-hyperstates on $\CA$. For $\psi\in\mathcal{S}_\tau(\CA)$, we naturally have $L^2(M,\tau)\subset L^2(\CA,\psi)$. For convention, we simply denote $L^2(M,\tau)$ by $L^2(M)$, and denote by $\hat{1}\in L^2(M)$ the cyclic vector and $\hat{z}=z\hat{1}$ for $z\in M$. Let $e_M\in B(L^2(\CA,\psi))$ be the orthogonal projection onto $L^2(M)$. The ucp map $\CP_\psi:\CA\to B(L^2(M))$ is defined as 
$$\CP_\psi(T)=e_MTe_M, \ T\in \CA.$$
Following \cite[Proposition 2.1]{DP20}, $\psi\mapsto\CP_\psi$ is a bijection between hyperstates on $\CA$ and ucp $M$-bimodular maps from $\CA$ to $B(L^2(M))$, whose inverse is $\CP \mapsto \langle\CP(\,\cdot\,)\hat{1},\hat{1}\rangle$. When $\CA$ is a von Neumann algebra, $\psi$ is normal if and only if $\CP_\psi$ is normal.

For $\psi\in\mathcal{S}_\tau(\CA)$ and $\varphi \in \mathcal{S}_\tau(B(L^2(M)))$, the convolution $\varphi \ast\psi\in \mathcal{S}_\tau(\CA)$ is defined to be the hyperstate associated with the $M$-bimodular ucp map $\mathcal{P}_{\varphi}\circ \mathcal{P}_{\psi}$. And $\psi$ is said to be $\varphi$-stationary if $\varphi \ast\psi=\psi$.

Let $\varphi\in\mathcal{S}_\tau(B(L^2(M)))$ be a hyperstate. The set of $\mathcal{P}_\varphi$\textbf{-harmonic operators} is defined to be
$$\mathrm{Har}(\mathcal{P}_\varphi)=\mathrm{Har}(B(L^2(M)),\mathcal{P}_\varphi)=\{ T\in B(L^2(M))\mid \mathcal{P}_\varphi(T)=T\}. $$
As an operator system, $\mathrm{Har}(\mathcal{P}_\varphi)$ is always injective, i.e., there exists a conditional expectation $E:\BL\to\mathrm{Har}(\mathcal{P}_\varphi)$.

The \textbf{noncommutative Poisson boundary} $\mathcal{B}_\varphi$ of $M$ with respect to $\varphi$ is defined to be the noncommutative Poisson boundary of the ucp map $\mathcal{P}_\varphi$ as defined by Izumi \cite{Izu02}, that is, the Poisson boundary $\mathcal{B}_\varphi$ is the unique $\mathrm{C}^*$-algebra (a von Neumann algebra when $\varphi$ is normal) that is isomorphic, as an operator system, to the space of harmonic operators $\mathrm{Har}(\mathcal{P}_\varphi)$. And the isomorphism $\CP:\CB_\varphi\to \mathrm{Har}(\mathcal{P}_\varphi)$ is called the $\varphi$-\textbf{Poisson transform}. Since $M\subset\mathrm{Har}(\mathcal{P}_\varphi)$, $M$ can also be embedded into $\CB_\varphi$ as a subalgebra. We said that $\CB_\varphi$ is \textbf{trivial} if $\CB_\varphi=M$. For the inclusion 
$M\subset \CB_\varphi$, $\zeta:=\varphi\circ\CP\in\mathcal{S}_\tau(\CB_\varphi)$ is the \textbf{canonical $\varphi$-stationary hyperstate} on $\CB_\varphi$.

Following \cite[Proposition 2.8]{DP20}, for a normal hyperstate $\varphi\in\mathcal{S}_\tau(B(L^2(M)))$, there exists a sequence $\{z_n\}\subset M$ such that $\sum_{n=1}^{\infty}z^*_nz_n=1$, and $\varphi$ and $\CP_\varphi$ admit the following standard form:
$$\varphi(T)=\sum_{n=1}^{\infty}\langle T  \hat{z}_n^*,\hat{z}_n^*\rangle, \ \mathcal{P}_\varphi(T)=\sum_{n=1}^{\infty} (Jz_n^*J)T(Jz_nJ), \ T\in B(L^2(M)).$$
Following \cite[Proposition 2.5 and the unnumbered remark right after Proposition 2.8]{DP20}, $\varphi$ is said to be 
\begin{itemize}
\item \textbf{regular}, if $\sum_{n=1}^{\infty}\limits z_n^*z_n=\sum_{n=1}^{\infty}\limits z_nz_n^*=1$;
\item \textbf{strongly generating}, if the unital algebra (rather than the unital $\ast$-algebra) generated by $\{z_n\}$ is weakly dense in $M$.
\end{itemize}
Following \cite[Proposition 2.9]{DP20}, when $\varphi$ is a normal regular strongly generating hyperstate, the canonical hyperstate $\zeta=\varphi\circ\CP$ is a normal faithful hyperstate on $\CB_\varphi$.

The following are some lemmas regarding ucp maps, which will be frequently used in this paper.

\begin{lemma}\label{stationary hyperstate}
For any C$^*$-inclusion $M\subset \CA$, $\CA$ admits a $\varphi$-stationary hyperstate, or equivalently, an $M$-bimodular ucp map from $\CA$ to $\har(\CP_\varphi)$ or $\CB_\varphi$.
\end{lemma}

\begin{proof}
By the Hahn-Banach theorem, $\CA$ admits a hyperstate $\eta$. Let $E$ be the conditional expectation from $\BL$ to $\har(\CP_\varphi)$. Then $E\circ\CP_\eta$ is an $M$-bimodular ucp map from $\CA$ to $\har(\CP_\varphi)$ and the associated hyperstate is $\varphi$-stationary.
\end{proof}

\begin{lemma}\label{ucp homomorphism}
Let $\CA,\CB,\mathcal{C}$ be C$^*$-algebras. Assume that two ucp maps $\CP_1:\CA\to\CB$ and $\CP_2:\CB \to \mathcal{C}$ satisfy that $\CP_2$ is faithful and $\CP_2\circ\CP_1$ is a $*$-homomorphism, then both $\CP_1$ and $\CP_2|_{\CP_1(\CA)}$ are $*$-homomorphisms. 
\end{lemma}
\begin{proof}
For any $T\in \CA$, by Kadison's inequality, we have
\begin{equation}
 \CP_2\circ \CP_1(T^*T)\geq \CP_2(\CP_1(T^*)\CP_1(T))\geq \CP_2(\CP_1(T^*))\CP_2(\CP_1(T))=\CP_2\circ \CP_1(T^*T).
\end{equation}
Since $\CP_2$ is faithful, we must have $\CP_1(T^*T)=\CP_1(T^*)\CP_1(T)$ and $\CP_2(\CP_1(T^*)\CP_1(T))=\CP_2(\CP_1(T^*))\CP_2(\CP_1(T))$. Therefore, both $\CP_1$ and $\CP_2|_{\CP_1(\CA)}$ are $*$-homomorphisms.
\end{proof}

\begin{lemma}\label{ucp on A_i to homomorphism}
Let $\CA,\CB$ be two C$^*$-algebras. Assume that $\CA$ can be generated by two C$^*$-subalgebras $\CA_1,\CA_2\subset \CA$. Then for any ucp map $\CP:\CA\to \CB$, we have 
\begin{itemize}
    \item [(1)] $\CP$ is a $*$-homomorphism if and only if both $\CP|_{\CA_1}$ and $\CP|_{\CA_2}$ are $*$-homomorphisms. 
    \item [(2)] Given a $*$-homomorphism $\Phi:\CA\to \CB$. Then $\CP=\Phi$ if and only if $\CP|_{\CA_i}=\Phi|_{\CA_i}$ for $i=1,2$.
\end{itemize}

\end{lemma}

\begin{proof}
(1)``$\Rightarrow$'' is clear.

``$\Leftarrow$'': Assume that both $\CP|_{\CA_1}$ and $\CP|_{\CA_2}$ are $*$-homomorphisms. Following \cite[Theorem 3.1]{Ch74}, both $\CA_1$ and $\CA_2$ are contained in the multiplicative domain of $\CP$. Since $\CA_1$ and $\CA_2$ generate $\CA$ as a C$^*$-algebra, $\CA$ is also contained in the multiplicative domain of $\CP$. Hence $\CP:\CA\to \CB$ is a $*$-homomorphism.

(2)``$\Rightarrow$'' is clear.

``$\Leftarrow$'': Assume that $\CP|_{\CA_i}=\Phi|_{\CA_i}$ for $i=1,2$. Then $\CP|_{\CA_i}$ $(i=1,2)$ is a $*$-homomorphism. By (1) we know that $\CP:\CA\to \CB$ is a $*$-homomorphism. Moreover, since $\CP|_{\CA_1\cup \CA_2}=\Phi|_{\CA_1\cup \CA_2}$ on a generating set $\CA_1\cup \CA_2$ of $\CA$, we must have $\CP=\Phi$.
\end{proof}

\subsection{The Effros-Mar\'{e}chal topology on \texorpdfstring{$\mathrm{vN}(H)$}{}.} Let $H$ be a separable Hilbert space. Following \cite{Eff65,Mar73,HW98,AHW16}, let $\mathrm{vN}(H)$ be the collection of von Neumann subalgebras of $B(H)$, define the \textbf{Effros-Mar\'{e}chal topology} on $\mathrm{vN}(H)$ as follows: For a sequence $(M_n)\subset\mathrm{vN}(H)$, let
$$\liminf_n M_n =\{x\in B(H)\mid \exists (x_n)\in \ell^\infty(\N,M_n) \mbox{ with } x=\mathrm{so}^*-\lim_n x_n\};$$
$$\limsup_n M_n =\lag\{x\in B(H)\mid  x \mbox{ is a wo-limit point of some } (x_n)\in \ell^\infty(\N,M_n)\}\rag.$$
Here $\lag\,\cdot \,\rag$ refers to the von Neumann algebra generated by the given set and the strong* operator topology is defined by $x_n\xrightarrow{\mathrm{so^*}}x$ iff $x_n\xrightarrow{\mathrm{so}}x$ and $x_n^*\xrightarrow{\mathrm{so}}x^*$. Then we say that $M_n\to M\in \mathrm{vN}(H)$ if
$$\liminf_n M_n=\limsup_n M_n=M.$$
Following \cite[Theorem 3.5]{HW98}, we have that 
$$\liminf_n M_n'=(\limsup_n M_n)',$$
$$\limsup_n M_n'=(\liminf_n M_n)'.$$
In particular, $M\in\mathrm{vN}(H)\mapsto M'\in\mathrm{vN}(H)$ is a homeomorphism.

For a countable discrete group $\Gamma$, let $\mathrm{Sub}(\Gamma)$ be the collection of subgroups of $\Gamma$. Define the \textbf{Chabauty topology} \cite{Cha50} on $\mathrm{Sub}(\Gamma)$ as follows: For a sequence $(\Lambda_n)\subset\mathrm{Sub}(\Gamma)$, let
$$\liminf_n \Lambda_n=\bigcup_{n=1}^{\infty}\bigcap_{k=n}^{\infty}\Lambda_k;$$
$$\limsup_n \Lambda_n=\bigcap_{n=1}^{\infty}\bigcup_{k=n}^{\infty}\Lambda_k.$$
Then we say that $\Lambda_n\to\Lambda$ if 
$$\liminf_n \Lambda_n=\limsup_n \Lambda_n=\Lambda.$$
\cite[Proposition 4.1]{AHO23} shows that the map $L: \Lambda\in\mathrm{Sub}(\Gamma)\mapsto L(\Lambda)\in \SA(L(\Gamma))$ is continuous.

It was shown in \cite[Theorem 1.4]{BDLW16} that for any $\Gamma$-invariant measure $\mu\in\prob(\mathrm{Sub}(\Gamma))$ with $\mu$-almost every $\Lambda\in\mathrm{Sub}(\Gamma)$ amenable, we must have $\Lambda\subset\rad(\Gamma)$ $\mu$-almost everywhere. Recently, \cite{AHO23} extended this result to group von Neumann algebras. In Sections \ref{IRA} and \ref{IRR}, we will extend this result to tracial crossed products and equivalence relations arising from free pmp actions.

\subsection{Jones' basic construction.}

Following \cite{Jo83} (see also \cite{AP17}), let $( M,\tau)$ be a tracial von Neumann algebra and $B\subset M$ be a von Neumann subalgebra. Let $e_B\in B(L^2( M))$ be the orthogonal projection onto $L^2(B)$. The von Neumann algebra $\langle M, e_B \rangle\subset B(L^2(M))$ generated by $M$ and $e_B$ is called the \textbf{Jones' basic construction} of $B\subset M$, which satisfies that $\langle M, e_B \rangle=JB'J$, where $J:L^2(M)\to L^2(M)$ is the modular conjugation operator of $(M,\tau)$. Since $B$ is amenable if and only if $B'$ is amenable, we know that $B$ is amenable if and only if $\langle M, e_B \rangle=JB'J$ is amenable.

\section{Noncommutative topological boundaries}\label{Section NC topological boundary}

In Section \ref{Section NC topological boundary}, \ref{section prox usb}, and \ref{section tight inclusion}, we fix a separable tracial von Neumann algebra $(M,\tau)$.

As an analogue of the equivalent conditions (1-3) for the definition of $\Gamma$-boundary in Subsection \ref{group boundary}, we have the following equivalent conditions (i-iii) in the noncommutative setting.
\begin{theorem}\label{eq def of boundary}
For a C$^*$-inclusion $M\subset \CA$, the following conditions are equivalent.
\begin{itemize}
    \item [(i)] There exists a completely isometric $\CP_0\in \UCP_M(\CA,\BL)$ and for any $\CP\in\UCP_M(\CA,\BL)$, one has $$\UCP_M(\BL,\BL)\circ\CP =\UCP_M(\CA,\BL)$$ (or equivalently, for any $\eta\in \CS_\tau(\CA)$, $\CS_\tau(\BL)\ast \eta=\CS_\tau(\CA)$);
    \item[(ii)] Any $\CP\in\UCP_M(\CA,\BL)$ is completely isometric;
    \item[(iii)] The inclusion $M\subset \CA$ is essential \cite[Definition 4.5]{Ham79}, i.e., for any other C$^*$-inclusion $M\subset \CB$, any $\CP\in\UCP_M(\CA,\CB)$ is completely isometric.
\end{itemize}
\begin{proof}
(i) $\Rightarrow$ (ii): Fix a $\CP\in\UCP_M(\CA,\BL)$. Since the complete isometry $\CP_0\in\UCP_M(\CA,\BL)=\UCP_M(\BL,\BL)\circ\CP$, there exists a $\CP_1\in \UCP_M(\BL,\BL)$ such that $\CP_1\circ\CP=\CP_0$. Since $\CP_0$ is completely isometric, for any $T\in M_n(\CA)$, we have
\begin{equation}\label{completely isometric}
 \Vert T\Vert=\Vert\CP_0(T)\Vert=\Vert\CP_1\circ\CP(T)\Vert\leq \Vert\CP(T)\Vert\leq\Vert T\Vert.   
\end{equation}
Hence $\Vert\CP(T)\Vert=\Vert T\Vert$ for any $T\in M_n(\CA)$ and $\CP$ is completely isometric.\\

(ii) $\Rightarrow$ (iii): Fix a C$^*$-inclusion $M\subset \CB$ and $\CP\in\UCP_M(\CA,\CB)$. Take $\CP_2\in \UCP_M(\CB,\BL)$, which is not empty by Lemma \ref{stationary hyperstate}. Then we have that $\CP_2\circ \CP \in \UCP_M(\CA,\BL)$ is completely isometric. By the same discussion in (\ref{completely isometric}), $\CP_2\circ \CP$ being completely isometric induces that $\CP$ is also completely isometric.\\

(iii) $\Rightarrow$ (i): By Lemma \ref{stationary hyperstate}, $\UCP_M(\CA,\BL)$ is not empty. For any $\CP\in\UCP_M(\CA,\BL)$, since $\CA$ satisfies (iii), $\CP$ is completely isometric. For any $\CP_3\in\UCP_M(\CA,\BL)$, since $\BL$ is injective as an operator system and $\CP: \CA\to\BL$ is an embedding between operator systems, there exists a ucp map $\CP_4:\BL\to\BL$ with $\CP_4\circ\CP=\CP_3$. 
$$
\begin{tikzcd}
B(L^2(M)) \arrow[rd, "\mathcal{P}_4", dashed]                         &           \\
\mathcal{A} \arrow[u, "\mathcal{P}", hook] \arrow[r, "\mathcal{P}_3"] & B(L^2(M))
\end{tikzcd}$$
Also, since both $\CP$ and $\CP_3$ fix $M$, we know that $\CP_4$ is $M$-bimodular and $$\CP_3=\CP_4\circ\CP\in\UCP_M(\BL,\BL)\circ\CP.$$ Hence we have
$$\UCP_M(\BL,\BL)\circ\CP =\UCP_M(\CA,\BL).$$
\end{proof}
\end{theorem}

\begin{definition}\label{def NC topological boundary}
For a C$^*$-inclusion $M\subset \CA$, we say that $\CA$ is a \textbf{$M$-boundary} if it satisfies one of the equivalent conditions (i)-(iii) in Theorem \ref{eq def of boundary}.
\end{definition}

The following example of a noncommutative topological boundary is a direct corollary of Theorem \ref{Gamma A boundary}, to which we refer for the proof.
\begin{example}
Let $\Gamma$ be a countable discrete group, $X$ be a $\Gamma$-boundary, and fix a point $x\in X$. By the $\Gamma$-minimality of $X$, we can embed $C(X)$ into $\ell^\infty(\Gamma)$ as a C$^*$-subalgebra through the map
$$P_{\delta_x}: f\in C(X)\mapsto(g\mapsto f(gx))\in\ell^\infty(\Gamma).$$
Then the C$^*$-algebra $\CB_X=C^*(L(\Gamma),C(X))\subset B(\ell^2(\Gamma))$ is an $L(\Gamma)$-boundary.
\end{example}

Now let us define the Furstenberg boundary in our setting. Boundaries defined as injective envelopes have appeared previously in various settings, where they are referred to as \textbf{Furstenberg--Hamana boundaries}; see, for example, \cite{KK14, BK21, Bor19, KKSV22}. We adopt the same approach in our setting:
\begin{definition}
We define the \textbf{$M$-Furstenberg boundary}, $\partial_F M$, to be the injective envelope $I(M)$ of $M$ \cite{Ham79}.
\end{definition}

The following proposition is \cite[Lemma 4.6]{Ham79}:
\begin{proposition}\label{boundary is subalg of F boundary}
Up to $M$-bimodular $*$-isomorphisms, any $M$-boundary is a C$^*$-subalgebra of $\partial_F M$ that contains $M$.
\end{proposition}

The following proposition is \cite[Proposition 4.8]{Ham79}:
\begin{proposition}
The von Neumann algebra $(M,\tau)$ is amenable if and only if $\partial_F M$ is trivial, i.e., $\partial_F M=M$.
\end{proposition}

\begin{proposition}\label{exists completely isometry}
Let $\CA$ be an $M$-boundary. Then for any C$^*$-inclusion $M\subset \CB$ such that $\CB$ is injective, there exists a completely isometric $\CP\in\UCP_M(\CA,\CB)$.
\end{proposition}
\begin{proof}
By \cite[Theorem 4.1]{Ham79}, since $\CB$ is injective, we have that $\partial_F M=I(M)$ can be viewed as a sub-operator system of $\CB$, i.e., there exists a completely isometric $\CP'\in\UCP_M(\partial_F M,\CB)$. By Proposition \ref{boundary is subalg of F boundary}, we can view $\CA$ as a subalgebra of $\partial_F M$, then $\CP=\CP'|_\CA\in\UCP_M(\CA,\CB)$ is completely isometric.
\end{proof}

\section{\texorpdfstring{$\varphi$}{}-proximality and \texorpdfstring{$\varphi$}{}-unique stationary noncommutative boundary}\label{section prox usb}

Let $\Gamma$ be a countable discrete group, $\mu\in\prob(\Gamma)$ be a generating measure and $(B,\nu_B)$ be the $\mu$-Poisson boundary. Following \cite{Fu63,Fu63a}, given a compact metrizable $(\Gamma,\mu)$-space $(Y,\nu_Y)$, there exists an (essentially) unique $\Gamma$-equivariant measurable map $\beta_{\nu_Y}:B\to\prob(Y)$ with $\nu_Y=\int_B\beta_{\nu_Y}(b)\d\nu_B(b)$. The map $\beta_{\nu_Y}$ is usually called \textbf{Furstenberg's boundary map}. Recall that a compact metrizable $\Gamma$-space $X$ is called \textbf{$\mu$-proximal} if for any $\mu$-stationary $\nu\in\prob(X)$, $(X,\nu)$ is a $\mu$-boundary, i.e., the Furstenberg’s boundary map $\beta_\nu:B\to\prob(X)$ satisfies that for $\nu_B$-a.e. $b\in B$, $\beta_\nu(b)=\delta_{\pi(b)}\in \prob(X)$ is a Dirac mass. 

Fix a separable tracial von Neumann algebra $(M,\tau)$. Following \cite[Theorem 3.2]{Zh23b}, for a C$^*$-inclusion $M\subset \CA$ and $\varphi$-stationary $\eta\in\CS_\tau(\CA)$, the \textbf{noncommutative Furstenberg’s boundary map} associated with $\eta$ is defined to be the unique $M$-bimodular ucp map $\Phi_\eta:\CA\to\CB_\varphi$ with $\eta=\zeta\circ\Phi_\eta$. Inspired by these facts, we are able to define the noncommutative analogue of $\mu$-proximality.

Fix a normal regular strongly generating hyperstate $\varphi\in \CS_\tau(\BL)$ and let $(\CB_\varphi,\zeta)$ be the $\varphi$-Poisson boundary.
\begin{definition}\label{def prox}
We say that a C$^*$-inclusion $M\subset \CA$ is \textbf{$\varphi$-proximal} if it satisfies one of the following equivalent conditions (see \cite[Theorem 3.2]{Zh23b} for the proof of equivalence):
\begin{itemize}
    \item [(1)] For any $\varphi$-stationary $\eta\in\CS_\tau(\CA)$, the noncommutative Furstenberg’s boundary map $\Phi_\eta:\CA\to\CB_\varphi$ is a faithful $*$-homomorphism;
    \item [(2)] Any $\CP\in\UCP_M(\CA,\CB_\varphi)$ is a faithful $*$-homomorphism.
\end{itemize}
\end{definition}
The following proposition is the noncommutative analogue of \cite[Corollary 2.10]{Mar91}.
\begin{proposition}\label{prox unique}
Let $M\subset \CA$ be a $\varphi$-proximal C$^*$-inclusion. For any C$^*$-inclusion $(M,\tau)\subset (\CB,\psi)$ with faithful $\varphi$-stationary $\psi\in \CS_\tau(\CB)$, we have
\begin{itemize}
    \item [(i)] Any $\CP\in \UCP_M(\CA,\CB)$ is a faithful $*$-homomorphism;
    \item [(ii)] $\UCP_M(\CA,\CB)$ admits at most one element.
    
\end{itemize}
\end{proposition}
\begin{proof}
(i) Fix a $\CP\in \UCP_M(\CA,\CB)$. Let $\Phi_\psi:(\CB,\psi)\to (\CB_\varphi,\zeta)$ be the noncommutative Furstenberg’s boundary map associated with $\psi$. Since $\psi$ is faithful, then so is $\Phi_\psi$. Since $\Phi_\psi\circ \CP\in \UCP_M(\CA,\CB_\varphi)$, it must be a faithful $*$-homomorphism. Moreover, since $\Phi_\psi$ is faithful, by Lemma \ref{ucp homomorphism}, $\CP$ is a $*$-homomorphism. Also, since $\Phi_\psi\circ \CP$ is faithful, then $\CP$ must be faithful.\\

(ii) For any $\CP_1,\CP_2\in \UCP_M(\CA,\CB)$, since $\frac{1}{2}(\CP_1+\CP_2)\in \UCP_M(\CA,\CB)$, by (i), $\frac{1}{2}(\CP_1+\CP_2)$ is a $*$-homomorphism. By \cite[proposition 4.14]{Ham79}, $\frac{1}{2}(\CP_1+\CP_2)$ is an extreme point of $\UCP_M(\CA,\CB)$ for being a $*$-homomorphism. Therefore, we must have $\CP_1=\CP_2$ and $\UCP_M(\CA,\CB)$ admits at most one element.
\end{proof}

The following corollary is the noncommutative analogue of \cite[Theorem 3.11]{HK23}.
\begin{corollary}\label{prox boundary state}
For any $\varphi$-proximal $M\subset \CA$, $\CA$ is an $M$-boundary that admits a unique $\varphi$-stationary hyperstate.
\end{corollary}
\begin{proof}
Take $(\CB,\psi)=(\CB_\varphi,\zeta)$ in Proposition \ref{prox unique}. Let $E:\BL\to \har(\CP_\varphi)$ be the conditional expectation and $\CP_\zeta:\CB_\varphi\to \har(\CP_\varphi)$ be the Poisson transform. Then by Proposition \ref{prox unique} (i), for any $\CP\in \UCP_M(\CA,\BL)$, $\CP_\zeta^{-1}\circ E\circ \CP\in \UCP_M(\CA,\CB_\varphi)$ is a faithful $*$-homomorphism, which is also completely isometric. By the same discussion as in (\ref{completely isometric}), $\CP$ must be completely isometric. Hence $\CA$ is an $M$-boundary.  

By Lemma \ref{stationary hyperstate} and Proposition \ref{prox unique} (ii), $\UCP_M(\CA,\CB_\varphi)$ admits exactly one element. Moreover, by the bijection between $\varphi$-stationary hyperstates and $\UCP_M(\CA,\CB_\varphi)$ \cite[Theorem 3.2]{Zh23b}, $\CA$ admits a unique $\varphi$-stationary hyperstate.
\end{proof}

Following \cite[Definition 3.9]{HK23}, for $\mu\in\prob(\Gamma)$, a $\mu$-boundary $(B_0,\nu_0)$ is a \textbf{$\mu$-unique stationary boundary} ($\mu$-USB) if it has a compact model $(\Bar{B}_0,\Bar{\nu}_0)$ such that $\Bar{\nu}_0$ is the unique $\mu$-stationary Borel probability measure on $\Bar{B}_0$. The $\mu$-USB has been greatly used in studying C$^*$-simplicity of groups \cite{HK23}.

Following \cite[Definition 3.7]{Zh23}, up to $*$-isomorphism, a \textbf{$\varphi$-boundary} $(\CB_0,\zeta_0)$ is a von Neumann subalgebra of $(\CB_\varphi,\zeta)$. Inspired by these facts, we are able to define the noncommutative analogue of $\mu$-USB.
\begin{definition}
We say that a $\varphi$-boundary $(\CB,\zeta|_\CB)\subset(\CB_\varphi,\zeta)$ is a \textbf{$\varphi$-unique stationary noncommutative boundary} ($\varphi$-USNCB) if it admits a weakly dense C$^*$-subalgebra ${\CA_0}\subset \CB$ with $M\subset {\CA_0}$ such that $\zeta|_{\CA_0}$ is the unique $\varphi$-stationary hyperstate on ${\CA_0}$.
\end{definition}

The following proposition establishes the link between the notions of $\varphi$-proximality and $\varphi$-USNCB.

\begin{proposition}\label{USB prox}
For a $\varphi$-USNCB $(\CB,\zeta|_\CB)$, the weakly dense C$^*$-subalgebra ${\CA_0}$ with unique $\varphi$-stationary hyperstate is $\varphi$-proximal. Conversely, for a $\varphi$-proximal inclusion $M\subset {\CA_0}$, let $\Phi: {\CA_0}\to \CB_\varphi$ be the unique $M$-bimodular ucp map, then the weak closure of $(\Phi({\CA_0}),\zeta|_{\Phi({\CA_0})})$ is a $\varphi$-USNCB.
\end{proposition}

\begin{proof}
Let $(\CB,\zeta|_\CB)$ be a $\varphi$-USNCB and  ${\CA_0}\subset \CB$ be the weakly dense C$^*$-subalgebra with unique $\varphi$-stationary hyperstate. Then by \cite[Theorem 3.2]{Zh23b}, we have $\UCP_M({\CA_0},\CB_\varphi)=\{\id_{{\CA_0}}\}$. Hence every element of $\UCP_M({\CA_0},\CB_\varphi)$ is a faithful $*$-homomorphism and $M\subset{\CA_0}$ is $\varphi$-proximal.\\

For a $\varphi$-proximal inclusion $M\subset {\CA_0}$, let $\Phi: {\CA_0}\to \CB_\varphi$ be the unique $M$-bimodular ucp map. Then by Proposition \ref{prox unique}, $\Phi$ is a faithful $*$-homomorphism. Hence $(\CB_0,\zeta|_{\CB_0})=\overline{(\Phi({\CA_0}),\zeta|_{\Phi({\CA_0})})}^{\mathrm{wo}}$ is a von Neumann subalgebra of $(\CB_\varphi,\zeta)$, which is a $\varphi$-boundary. By Corollary \ref{prox boundary state}, ${\CA_0}$ admits a unique $\varphi$-stationary hyperstate and so does $\Phi({\CA_0})$. Therefore, $\Phi({\CA_0})\subset \CB_0$ is a weakly dense C$^*$-subalgebra with unique $\varphi$-stationary hyperstate. Hence $(\CB_0,\zeta|_{\CB_0})$ is a $\varphi$-USNCB.
\end{proof}

\begin{example}
Let $M=L(\Gamma)$ and $\varphi (T)=\sum_{\gamma\in\Gamma}\mu(\gamma)\langle T\mathds{1}_{\gamma},\mathds{1}_{\gamma}\rangle$ for a generating $\mu\in\prob(\Gamma)$. Assume that $(B_0,\nu_0)$ is a $\mu$-USB. Then the $L(\Gamma\car B_0)$ is a $\varphi$-USNCB.
\end{example}
\begin{proof}
We may assume that $B_0$ is a compact metrizable $\Gamma$-space with a unique $\mu$-stationary probability Borel measure. Let $(B,\nu_B)$ be the $\mu$-Poisson boundary. Then $\id_{C(B_0)}: C(B_0)\hookrightarrow  L^\infty(B)$ is the unique $\Gamma$-equivariant ucp map  from $C(B_0)$ to $ L^\infty(B)$. Let ${\CA_0}\subset L(\Gamma\car B_0)$ be the C$^*$-subalgebra generated by $L(\Gamma)$ and $C(B_0)$. Then $\CA_0$ is weakly dense in $L(\Gamma\car B_0)$ and we only need to prove that ${\CA_0}$ admits a unique $M$-bimodular ucp map to $\CB_\varphi=L(\Gamma\car B)$ \cite[Theorem 4.1]{Izu04}.

Fix a ucp map $\Phi\in\UCP_M({\CA_0}, L(\Gamma\car B))$. Let $E:L(\Gamma\car B)\to  L^\infty(B)$ be the canonical $\Gamma$-equivariant conditional expectation. Then $E\circ\Phi|_{C(B_0)}:C(B_0)\to  L^\infty(B)$ is also a $\Gamma$-equivariant ucp map. Hence we must have ${E\circ\Phi|_{C(B_0)}}=\id_{C(B_0)}$, which is a $*$-homomorphism. Since $E$ is faithful, by Lemma \ref{ucp homomorphism}, we must have both $\Phi|_{C(B_0)}$ and $E|_{\Phi(C(B_0))}$ are $*$-homomorphisms. Therefore, $\Phi(C(B_0))$ is contained in the multiplicative domain of $E$, which is $ L^\infty(B)$. Hence we must have $\Phi|_{C(B_0)}=E\circ\Phi|_{C(B_0)}=\id_{C(B_0)}$.

Moreover, since $\Phi|_{C(B_0)}=\id_{C(B_0)}$ and $\Phi|_{L(\Gamma)}=\id_{L(\Gamma)}$ are both $*$-homomorphisms and $\CA_0=\mathrm{C}^*(L(\Gamma),C(B_0))$, by Lemma \ref{ucp on A_i to homomorphism}, $\Phi$ must be a $*$-homomorphism and $\Phi=\id_{\CA_0}$. Therefore, the ucp map $\Phi: {\CA_0} \to L(\Gamma\car B)$ is unique and $L(\Gamma\car B_0)$ is a $\varphi$-USNCB.
\end{proof}

\section{\texorpdfstring{$M$}{}-tight inclusions}\label{section tight inclusion}
In this section, we extend the notion of tight inclusions and some of the related results in \cite{HK24} to the noncommutative setting. 

Following \cite{HK24}, for a countable discrete group $\Gamma$, an inclusion of $\Gamma$-C$^*$-algebras $A\subset B$ is \textbf{$\Gamma$-tight} if $\UCP_\Gamma(A,B)=\{\id_A\}$. Here $\UCP_\Gamma(A,B)$ is the set of $\Gamma$-equivariant ucp maps from $A$ to $B$. We still fix a separable tracial von Neumann algebra $(M,\tau)$. Then we have the following definition as a noncommutative analogue of $\Gamma$-tight inclusion.
\begin{definition}
We say that a C$^*$-inclusion $\CA\subset\CB$ is a \textbf{$M$-inclusion} if $M\subset\CA\subset\CB$. We say that an $M$-inclusion $\CA\subset\CB$ is \textbf{$M$-tight} if 
$$\UCP_M(\CA,\CB)=\{\id_\CA\}.$$
\end{definition}

The following two propositions show that tightness can be passed to intermediate subalgebras and C$^*$-extensions with faithful conditional expectation, which are the analogues of \cite[Proposition 2.2 and Lemma 2.12]{HK24}.
\begin{proposition}
Assume that $\CA\subset\CB$ is an $M$-tight inclusion, then for any C$^*$-algebra $\mathcal{C}$ with $\CA\subset\mathcal{C}\subset\CB$, the inclusion $\CA\subset\mathcal{C}$ is also $M$-tight.
\end{proposition}
\begin{proof}
Since
$$\UCP_M(\CA,\mathcal{C})\subset\UCP_M(\CA,\CB)=\{\id_\CA\},$$
we must have $\UCP_M(\CA,\mathcal{C})=\{\id_\CA\}$.
\end{proof}

\begin{proposition}
Assume that $\CA\subset\CB\subset\mathcal{C}$ are $M$-inclusions and $\CA\subset\CB$ is $M$-tight. If there exists a faithful conditional expectation $E:\mathcal{C}\to\CB$, then $\CA\subset\mathcal{C}$ is also $M$-tight.
\end{proposition}
\begin{proof}
Fix a $\CP\in\UCP_M(\CA,\mathcal{C})$. Only need to prove $\CP=\id_\CA$. We have 
$$E\circ\CP\in\UCP_M(\CA,\CB)=\{\id_\CA\}.$$
Hence $E\circ\CP=\id_\CA$. By Lemma \ref{ucp homomorphism}, $\CP(\CA)$ is contained in the multiplicative domain of $E$, i.e., $\CB$. Therefore, we have $\CP=E\circ\CP=\id_\CA$, which finishes the proof.
\end{proof}

The following proposition is the noncommutative analogue of \cite[Corollary 2.4]{HK24}.
\begin{proposition}
If there exists an $M$-tight inclusion $\CA\subset\CB$ with $M\subsetneqq\CA$, then $M$ is not amenable.
\end{proposition}
\begin{proof}
Assume that $M$ is amenable. Then there exists a conditional expectation $E:\CA\to M$. Then we have 
$$E\in\UCP_M(\CA,M)\subset\UCP_M(\CA,\CB)=\{\id_\CA\}.$$
Hence, we must have $E=\id_\CA$ and $\CA=M$, contradiction.
\end{proof}

\begin{proposition}\label{M' cap B}
Assume that $\CA\subset\CB$ is an $M$-tight inclusion, then 
$$M'\cap\CB=\CA'\cap\CB.$$
In particular, if the inclusion $\CA\subset\CB$ is irreducible (i.e., $\CA'\cap\CB=\C\cdot 1$), then so is $M\subset\CB$.
\end{proposition}
\begin{proof}
For any $u\in\mathcal{U}(M'\cap\CB)$, we have 
$$\mathrm{Ad}(u)|_\CA\in\UCP_M(\CA,\CB)=\{\id_A\}.$$
Hence $\mathrm{Ad}(u)|_\CA=\id_\CA$ and $u\in\CA'\cap\CB$.
\end{proof}

The following proposition is the noncommutative analogue of \cite[Proposition 2.7]{HK24}.
\begin{proposition}
For a C$^*$-inclusion $M\subset\CB$, $\CB$ admits a maximal C$^*$-subalgebra $\CA$ such that $\CA\subset\CB$ is $M$-tight.
\end{proposition}
\begin{proof}
Note that $M\subset\CB$ is $M$-tight. Then we can apply Zorn's Lemma.
\end{proof}

Here are some examples of $M$-tight inclusions.
\begin{example}
The von Neumann algebra $M$ is a maximal C$^*$-subalgebra of $B(L^2(M))$ such that its inclusion into $B(L^2(M))$ is $M$-tight.
\end{example}
\begin{proof}
Assume that $\CA\subset B(L^2(M))$ is $M$-tight. Then by Proposition \ref{M' cap B}, we have $M'=\CA'$ within $B(L^2(M))$, which is equivalent to $\CA''=M$. Moreover, since $M\subset\CA$, we must have $\CA=M$.
\end{proof}

The following example presents the rigidity of noncommutative topological boundaries from the perspective of tight inclusions, which is a noncommutative analogue of \cite[ Example 2.9]{HK24}.
\begin{example}
Let $\CA$ be an $M$-boundary and view it as a subalgebra of $\partial_F M$, then $\CA\subset\partial_F M$ is $M$-tight.
\end{example}
\begin{proof}
Fix $\CP\in\UCP_M(\CA,\partial_F M)$. Only need to prove $\CP=\id_\CA$. Since $\partial_F M$ is injective, there exists a ucp map $\CP_0:\partial_F M\to\partial_F M$ with $\CP_0|_\CA=\CP$.
$$
\begin{tikzcd}
\partial_F M \arrow[rd, "\mathcal{P}_0", dashed]                         &           \\
\CA \arrow[u, "\id_\CA", hook] \arrow[r, "\CP"] & \partial_F M
\end{tikzcd}$$
Hence $\CP_0|_M=\CP|_M=\id_M$. By \cite[Lemma 3.7]{Ham79}, we must have $\CP_0=\id_{\partial_F M}$ and $\CP=\CP_0|_\CA=\id_\CA$.
\end{proof}

The following example provides an equivalent definition of $\varphi$-proximality using the approach of tight inclusions.
\begin{example}
Let $\varphi\in\CS_\tau(B(L^2(M)))$ be a normal regular strongly generating hyperstate and $(\CB_\varphi,\zeta)$ be the $\varphi$-Poisson boundary. Then a C$^*$-inclusion $M\subset\CA$ is $\varphi$-proximal if and only if we have $\CA\subset\CB_\varphi$ as a C$^*$-subalgebra and $\CA\subset\CB_\varphi$ is $M$-tight.
\end{example}
\begin{proof}
Assume that $M\subset\CA$ is $\varphi$-proximal. Let $\Phi:\CA\to\CB_\varphi$ be the unique $M$-bimodular ucp map from $\CA$ to $\CB_\varphi$. Then by Proposition \ref{prox unique} (i), $\Phi$ is a faithful $\ast$-homomorphism. Hence we have $\CA\cong\Phi(\CA)\subset\CB_\varphi$ as a C$^*$-subalgebra. Moreover, by Proposition \ref{prox unique} (ii), $\CA\subset\CB_\varphi$ is $M$-tight.

Assume that $\CA\subset\CB_\varphi$ is $M$-tight C$^*$-inclusion. Then $\UCP_M(\CA,\CB_\varphi)=\{\id_\CA\}$. By Definition \ref{def prox}, $M\subset\CA$ is $\varphi$-proximal.
\end{proof}

\begin{example}
Let $\Gamma\car N$ be a $\Gamma$-W$^*$-dynamic system. Assume that $N_0\subset N$ is a $\Gamma$-invariant C$^*$-subalgebra such that $N_0\subset N$ is $\Gamma$-tight. Then $\mathrm{C}^*(L(\Gamma),N_0)\subset\Gamma\ltimes N$ is $L(\Gamma)$-tight.
\end{example}
\begin{proof}
Let $\mathcal{N}_0=\mathrm{C}^*(L(\Gamma),N_0)$. Fix an $L(\Gamma)$-bimodular ucp map $\Phi:\mathcal{N}_0\to\Gamma\ltimes N$. Only need to prove $\Phi=\id_{\mathcal{N}_0}$.

Let $E:\Gamma\ltimes N\to N$ be the canonical $\Gamma$-equivariant conditional expectation. Then $E\circ\Phi|_{N_0}\in\UCP_\Gamma(N_0,N)$. Since $N_0\subset N$ is $\Gamma$-tight, we must have $E\circ\Phi|_{N_0}=\id_{N_0}$. By Lemma \ref{ucp homomorphism}, we have that $\Phi({N_0})$ is contained in the multiplicative domain of $E$, which is $N$. Hence we have $\Phi|_{N_0}=E\circ\Phi|_{N_0}=\id_{N_0}$. Moreover, since $\Phi|_{L(\Gamma)}=\id_{L(\Gamma)}$, by Lemma \ref{ucp on A_i to homomorphism}, we have $\Phi=\id_{\mathcal{N}_0}$, which finishes the proof.
\end{proof}

Following \cite[Definition 2.14]{HK24}, a $\Gamma$-C$^*$-algebra $B$ is \textbf{$\Gamma$-weakly Zimmer amenable} if for any other $\Gamma$-C$^*$-algebra $A$, one has $\UCP_\Gamma(A,B)\not=\emptyset$. We have the following definition as a noncommutative analogue of weakly Zimmer amenability.
\begin{definition}
For a C$^*$-inclusion $M\subset\CB$, we say that $\CB$ is \textbf{$M$-weakly Zimmer amenable} if for any other C$^*$-inclusion $M\subset\CA$, one has 
$$\UCP_M(\CA,\CB)\not=\emptyset.$$
\end{definition}

The following proposition shows that injectivity (or amenability for von Neumann algebras) is stronger than weak Zimmer amenability. However, under certain conditions, they can be equivalent, as seen in Corollary \ref{weak amenable = amenable}.
\begin{proposition}\label{inj weakly amenable}
For any C$^*$-inclusion $M\subset\CB$ with $\CB$ injective, $\CB$ is $M$-weakly Zimmer amenable.
\end{proposition}
\begin{proof}
Fix a C$^*$-inclusion $M\subset\CA$. Only need to prove $\UCP_M(\CA,\CB)\not=\emptyset$. Since $\CB$ is injective as an operator system, there exists a ucp map $\CP:\CA\to\CB$ such that $\CP|_M=\id_M$.
$$\begin{tikzcd}
\mathcal{A} \arrow[rd, "\mathcal{P}", dashed]                         &           \\
M \arrow[u, "\id_M", hook] \arrow[r, "\id_M"] & \CB
\end{tikzcd}$$
Hence 
$$\CP\in\UCP_M(\CA,\CB)\not=\emptyset.$$
\end{proof}

The following propositions are the noncommutative analogue of \cite[Proposition 2.15 and 3.6]{HK24}
\begin{proposition}\label{weakly amenable and FB}
For a C$^*$-inclusion $M\subset\CB$, $\CB$ is $M$-weakly Zimmer amenable if and only if $\partial_F M\subset\CB$ as a sub-operator system that contains $M$.
\end{proposition}
\begin{proof}
Assume that $\CB$ is $M$-weakly Zimmer amenable. Then there exists a $\CP\in\UCP_M(\partial_F M,\CB)$. Since $M\subset\partial_F M$ is essential, we have that $\CP$ is completely isometric. Hence $\partial_F M\cong \CP(\partial_F M)\subset \CB$ as a sub-operator system that contains $M$.

Assume that $M\subset\partial_F M\subset\CB$. Then for any other C$^*$-inclusion $M\subset\CA$, since $\partial_F M$ is injective, by Proposition \ref{inj weakly amenable}, we have 
$$\emptyset\not=\UCP_M(\CA,\partial_F M)\subset\UCP_M(\CA,\CB).$$ 
Therefore, $\UCP_M(\CA,\CB)\not=\emptyset$ and $\CB$ is $M$-weakly Zimmer amenable.
\end{proof}

\begin{proposition}
Let $\CA\subset\CB$ be an $M$-tight inclusion. If $\CB$ is $M$-weakly Zimmer amenable, then $\CA$ is an $M$-boundary.
\end{proposition}
\begin{proof}
By Proposition \ref{ucp on A_i to homomorphism}, we have $\partial_F M\subset\CB$ as an operator system. By Proposition \ref{weakly amenable and FB}, $\partial_F M$ is $M$-weakly Zimmer amenable. Hence, there exists a 
$$\CP\in\UCP_M(\CA,\partial_F M)\subset\UCP_M(\CA,\CB)=\{\id_\CA\}.$$
Hence, we have $\CP=\id_\CA$ and $\CA\subset\partial_F M$ as an inclusion between operator systems. Since $\partial_F M$ is injective as an operator system, let $E: \CB\to \partial_F M$ be the conditional expectation. Then the C$^*$-algebra structure of $\partial_F M$ can be given by $x\circ y=E(xy)$ $(x,y\in\partial_F M)$. For $x,y\in \CA$, we have $x\circ y=E(xy)=xy$. Hence $\CA\cong(\CA,\circ)\subset(\partial_F M,\circ)$ is a C$^*$-inclusion. Therefore, $\CA$ is a subalgebra of $\partial_F M$, i.e., an $M$-boundary.
\end{proof}

Let $M\subset\CB$ be a C$^*$-inclusion (resp. W$^*$-inclusion). For C$^*$-subalgebras $\CA,\mathcal{C}\subset\CB$ that contain $M$, we write $\CB=\CA\vee\mathcal{C}$ if $\CB$ is the C$^*$-algebra (resp. von Neumann algebra) generated by $\CA$ and $\mathcal{C}$.

The following definition is the noncommutative analogue of \textbf{$\Gamma$-co-tight inclusions} in \cite[Definition 4.1]{HK24}.

\begin{definition}
We say that an $M$-inclusion $\CA\subset\CB$ is \textbf{$M$-co-tight} if there exists a C$^*$-algebra $\mathcal{C}$ with $M\subset\mathcal{C}\subset\CB$ such that $\mathcal{C}\subset\CB$ is $M$-tight and $\CB=\CA\vee\mathcal{C}$.
\end{definition}

The following propositions show the rigidity of co-tight inclusions, which extend \cite[Proposition 4.4, Lemma 4.5 and 4.6]{HK24} to the noncommutative setting.
\begin{proposition}
Let $M\subset\CB$ be a C$^*$-inclusion (resp. W$^*$-inclusion) and $\CA\subset\CB$ be an $M$-co-tight inclusion. Then for any (resp. normal) ucp map $\CP:\CB\to\CB$ with $\CP|_\CA=\id_\CA$, we must have $\CP=\id_\CB$.
\end{proposition}
\begin{proof}
Let $\mathcal{C}\subset\CB$ be an $M$-tight inclusion with $\CB=\CA\vee\mathcal{C}$. Then we must have $\CP|_\mathcal{C}=\id_\mathcal{C}$. Moreover, since $\CB=\CA\vee\mathcal{C}$ and $\CP|_\CA=\id_\CA$, we must have $\CP=\id_\CB$.
\end{proof}

\begin{proposition}\label{cotight UCP}
Let $M\subset\CB$ be a C$^*$-inclusion and $\CA\subset\CB$ be an $M$-co-tight inclusion. If there exists a ucp map $\CP\in\UCP_M(\CB,\CA)$, we must have $\CB=\CA$.
\end{proposition}
\begin{proof}
Let $\mathcal{C}\subset\CB$ be an $M$-tight inclusion with $\CB=\CA\vee\mathcal{C}$. Then we must have $\CP|_\mathcal{C}=\id_\mathcal{C}$ and $\mathcal{C}\subset \CA$. Therefore, $\CB=\CA\vee\mathcal{C}=\CA$.
\end{proof}

\begin{corollary}
Let $M\subset\CB$ be a C$^*$-inclusion and $\CA\subset\CB$ be an $M$-co-tight inclusion. Then for any $M$-weakly Zimmer amenable $\mathcal{C}$ with $\CA\subset\mathcal{C}\subset\CB$, we must have $\mathcal{C}=\CB$.
\end{corollary}
\begin{proof}
Since $\CA\subset\CB$ is $M$-co-tight, we know that $\mathcal{C}\subset\CB$ is also $M$-co-tight. Moreover, since $\mathcal{C}$ is $M$-weakly Zimmer amenable, there exists a ucp map $\CP\in\UCP_M(\CB,\mathcal{C})$. By Proposition \ref{cotight UCP}, we have $\mathcal{C}=\CB$.
\end{proof}

\section{Amenable invariant random intermediate subalgebras of crossed products}\label{IRA}
In this section, as an application of noncommutative topological boundaries, we generalize the main results of \cite{AHO23}.

 Let $\sigma:\Gamma \car (A,\tau_A)$ be a trace preserving action of the countable discrete group $\Gamma$ on a separable tracial von Neumann algebra $(A,\tau_A)$. Let $H=\ell^2(\Gamma,L^2(A))=\ell^2(\Gamma)\otimes L^2(A)$. Then we have $L(\Gamma)\otimes 1\subset B(H)=B(\ell^2(\Gamma))\otimes B(L^2(A))$. Define a representation $\pi:A\to B(H)$ by 
$$ (\pi(a) \xi)(h)=\sigma^{-1}_{h}(a)\xi(h)\in L^2(A)$$
for $\xi\in H=\ell^2(\Gamma,L^2(A))$, $a\in A$ and $h\in \Gamma$. 

Let
$$(M,\tau)=\Gamma \ltimes A=\lag L(\Gamma)\otimes 1, \pi(A)\rag \subset B(H)$$
be the crossed product of $\sigma:\Gamma \car (A,\tau_A)$, together with the canonical trace. Then $L^2(M)=H$, with cyclic vector $\xi_\tau=\mathds{1}_e\otimes\hat{1}_A$ and $(\lambda_g \otimes 1) \pi(a) \xi_\tau=  \mathds{1}_g\otimes\hat{a}$ for $g\in \Gamma$ and $a\in A$. And the modular conjugation operator $J$ of $(M,\tau)$ satisfies that for $g\in\Gamma$ and $a\in A$,
$$J(\lambda_g\otimes 1) J= \rho_g\otimes u_g ,$$
$$J\pi(a)J=1\otimes (J_A a J_A).$$
Here $\rho$ is the right regular representation of $\Gamma$, $J_A$ is the modular conjunction operator of $(A,\tau_A)$, and $u_g\in \mathcal{U}(L^2(A))$ is given by $u_g(\hat{a})=\widehat{\sigma_g(a)}$, which satisfies $u_g a u_g^*=\sigma_g(a)$ for any $a\in A$.

We have $\ell^\infty(\Gamma)\otimes 1\subset B(H)=B(\ell^2(\Gamma))\otimes B(L^2(A))$. For any subgroup $\Lambda<\Gamma$, we also have $\ell^\infty(\Gamma/\Lambda)\subset\ell^\infty(\Gamma)$ as a von Neumann subalgebra by identifying $f\in \ell^\infty(\Gamma/\Lambda)$ with $\sum_{C\in \Gamma/\Lambda} f(C)\cdot \mathds{1}_{C}\in \ell^\infty(\Gamma)$. Then we have the following lemma:

\begin{lemma}\label{crossed product as basic construction}
For any subgroup $\Lambda<\Gamma$, we have 
$$\lag M , \ell^\infty(\Gamma/\Lambda)\otimes 1\rag=\lag M,\mathds{1}_{\Lambda}\otimes 1\rag=J(\Lambda\ltimes A) J'.$$
\end{lemma}
\begin{proof}
Since 
$$\ell^\infty(\Gamma/\Lambda)=\lag \mathds{1}_{g\Lambda}\rag_{g\in\Gamma}=\lag \lambda_g\mathds{1}_{\Lambda}\lambda_g^*\rag_{g\in\Gamma}\subset\lag L(\Gamma),\mathds{1}_{\Lambda}\rag,$$
we have
$$\lag M , \ell^\infty(\Gamma/\Lambda)\otimes 1\rag=\lag M,\mathds{1}_{\Lambda}\otimes 1\rag=\lag M,e_{\ell^2(\Lambda)\otimes L^2(A)}\rag,$$
which is exactly the basic construction of $\Lambda\ltimes A \subset M$ because $$\ell^2(\Lambda)\otimes L^2(A)=\overline{\mathrm{span}\{\mathds{1}_g\otimes \hat{a}=(\lambda_g\otimes 1)\pi(a)\xi_\tau| g\in \Lambda, a\in A\}}=\overline{(\Lambda\ltimes A) \xi_\tau}.$$ Therefore,
$$\lag M , \ell^\infty(\Gamma/\Lambda)\otimes 1\rag=J(\Lambda\ltimes A) J'.$$
\end{proof}
Let $\mathrm{Sub}(\Gamma)$ be the collection of subgroups of $\Gamma$, endowed with the Chabauty topology, and $\SA(\Gamma\ltimes A)$ be the collection of subalgebras of $\Gamma$, endowed with the Effros-Mar\'echal topology. Then the following corollary generalizes \cite[Proposition 4.1]{AHO23}.
\begin{corollary}
The map $L_A:\mathrm{Sub}(\Gamma)\to \SA(\Gamma\ltimes A )$ defined by 
$$L_A(\Lambda)=\Lambda\ltimes A \ (\Lambda\in \mathrm{Sub}(\Gamma))$$
is continuous.
\end{corollary}
\begin{proof}
Assume that $\Lambda_n\to\Lambda$ within $\mathrm{Sub}(\Gamma)$. Only need to prove 
$$\limsup_n \Lambda_n\ltimes A \subset \Lambda\ltimes A \subset\liminf_n \Lambda_n\ltimes A .$$ 

For any $g\in\Lambda$, since $g\in \Lambda_n$ eventually, i.e., $\exists n_0\in \N$ with $g\in\Lambda_n$ for any $n\geq n_0$, we have $\lambda_g\in\liminf L(\Lambda_n)\subset\liminf_n \Lambda_n\ltimes A $. Hence $L(\Lambda)\subset\liminf_n \Lambda_n\ltimes A $. Moreover, since $A=\lim_n A\subset\liminf_n \Lambda_n\ltimes A $, we have
$$\Lambda\ltimes A =\lag L(\Lambda),A\rag\subset \liminf_n \Lambda_n\ltimes A .$$

By \cite[Theorem 3.5 and Corollary 3.6]{HW98}, $N\mapsto JN'J$ is a homeomorphism on $\mathrm{vN}(H)$, hence to prove $\limsup_n \Lambda_n\ltimes A \subset \Lambda\ltimes A $ is equivalent to prove 
$$J(\Lambda\ltimes A) 'J\subset\liminf_nJ(\Lambda_n\ltimes A) 'J.$$
And by Lemma \ref{crossed product as basic construction}, this is equivalent to prove
$$\lag M, \mathds{1}_\Lambda\otimes 1\rag\subset\liminf_n \lag M, \mathds{1}_{\Lambda_n}\otimes 1\rag.$$
Obviously, $M=\lim_n M\subset\liminf_n \lag M, \mathds{1}_{\Lambda_n}\otimes 1\rag$. Since $\Lambda_n\to \Lambda$, we have 
$$\mathds{1}_{\Lambda}\otimes 1=\mathrm{so}^*-\lim_n \mathds{1}_{\Lambda_n}\otimes 1\in \liminf_n \lag M, \mathds{1}_{\Lambda_n}\otimes 1\rag.$$
Therefore, $\lag M, \mathds{1}_\Lambda\otimes 1\rag\subset\liminf_n \lag M, \mathds{1}_{\Lambda_n}\otimes 1\rag$, which finishes the proof.
\end{proof}

Let 
$$\mathscr{A}=\lag M, \ell^\infty(\Gamma)\otimes 1\rag\subset B(H).$$
Then by taking $\Lambda=\{e\}$ in Lemma \ref{crossed product as basic construction}, we have
$$\mathscr{A}=\lag M, \ell^\infty(\Gamma)\otimes 1\rag=J\pi(A)J'=(1\otimes (J_A AJ_A))'=B(\ell^2(\Gamma))\otimes A.$$

\begin{theorem}\label{Gamma A boundary}
With the notations above, let $X$ be a $\Gamma$-boundary and take $x\in X $. Let $\CB_X^x \subset \mathscr{A}$ be the C$^*$-algebra generated by $M$ and $P_{\delta_x}(C(X))\otimes 1\subset \ell^\infty(\Gamma)\otimes 1$. Then we have
\begin{itemize}
    \item [(1)] Any $\CP\in \UCP_M(\CB_X^x ,\mathscr{A})$ is completely isometric;
    \item[(2)] $\CB_X^x $ does not depend on the choice of $x\in X$, i.e., for any $y\in X$, define $\CB_X^y =\mathrm{C}^*(M,P_{\delta_y}(C(X))\otimes 1)$ as above, then there exists an $M$-bimodular isomorphism $\Phi: \CB_X^x \to \CB_X^y $. Therefore, we can simply denote $\CB_X^x$ by $\CB_X$;
    \item[(3)] When $(A,\tau_A)$ is amenable, $\CB_X $ is an $M$-boundary. In particular, by taking $A=\C$, we have that $\CB_X=\mathrm{C}^*(L(\Gamma),C(X))\subset B(\ell^2(\Gamma))$ is an $L(\Gamma)$-boundary.
\end{itemize} 
\end{theorem}
\begin{proof}
(1) Fix a $\CP\in \UCP_M(\CB_X^x ,\mathscr{A})$. Let $E:B(\ell^2(\Gamma))\to \ell^\infty(\Gamma)$ be the canonical conditional expectation. Define a conditional expectation 
$$E_0=E\otimes \tau_A:\mathscr{A}=B(\ell^2(\Gamma))\otimes A\to \ell^\infty(\Gamma)\otimes 1.$$
Then $E_0$ is normal faithful and $\Gamma$-equivariant since both $E$ and $\tau_A$ are.

Obviously, $P_{\delta_x}:C(X)\to \ell^\infty(\Gamma)$ is a $*$-homomorphism. Moreover, since $X$ is $\Gamma$-minimal, $P_{\delta_x}$ is faithful. We can identify $C(X)$ with $P_{\delta_x}(C(X))\otimes 1$. Since 
$$E_0\circ\CP|_{C(X)}:C(X)\to \ell^\infty(\Gamma)\otimes 1\cong \ell^\infty(\Gamma)$$
is a $\Gamma$-equivariant ucp map, there exists a $\nu\in\prob(X)$ with 
$$E_0\circ\CP(f) =P_\nu(f)\otimes 1\in \ell^\infty(\Gamma)\otimes 1$$ 
for any $f\in C(X)$. Since $X$ is a $\Gamma$-boundary, there exists a net $(g_i)_I\subset \Gamma$ with $g_i\nu \xrightarrow{\mathrm{w}^*} \delta_x$. Then we have $\mathrm{Ad}(\rho_{g_i})\circ P_\nu=P_{g_i\nu}\to P_{\delta_x}$ pointwise weakly.

Since $\mathscr{A}=B(\ell^2(\Gamma))\otimes A$ is $\mathrm{Ad}(\rho_g\otimes u_g)$-invariant for any $g\in \Gamma$, take $\omega\in \beta I\setminus I$, we can define 
$$\CP_\omega=\lim_\omega \mathrm{Ad}(\rho_{g_i}\otimes u_{g_i})\in \UCP_M(\mathscr{A},\mathscr{A}).$$
Note that both $(E\otimes \id) $ and $ (\id\otimes \tau_A)$ commutes with $\mathrm{Ad}(\rho_{g}\otimes u_{g})$ (here uses the fact that $\sigma$ is $\tau_A$-preserving), hence so is $E_0$. We have 
\begin{equation}\label{E_0PP=P}
\begin{aligned}
&E_0\circ \CP_\omega\circ \CP|_{C(X)}\\
=&\lim_\omega E_0\circ \mathrm{Ad}(\rho_{g_i}\otimes u_{g_i})\circ \CP|_{C(X)} &(\mbox{$E_0$ is normal})\\
=&\lim_\omega \mathrm{Ad}(\rho_{g_i}\otimes u_{g_i})\circ E_0\circ \CP|_{C(X)} &(\mbox{$E_0$ commutes with $\mathrm{Ad}(\rho_{g}\otimes u_{g})$})\\
=&\lim_\omega \mathrm{Ad}(\rho_{g_i}\otimes u_{g_i})\circ (P_\nu\otimes1) &(\mbox{$E_0\circ \CP|_{C(X)}=P_\nu\otimes 1$})\\
=&\lim_\omega (\mathrm{Ad}(\rho_{g_i})\circ P_\nu)\otimes 1 &\\
=& P_{\delta_x}\otimes 1, &(\mbox{$\mathrm{Ad}(\rho_{g_i})\circ P_\nu=P_{g_i\nu}\to P_{\delta_x}$})
\end{aligned}
\end{equation}
 which is a $*$-homomorphism. By Lemma \ref{ucp homomorphism}, we know that $\CP_\omega\circ\CP|_{C(X)}$ is a $*$-homomorphism and $\CP_\omega\circ\CP(C(X))$ is contained in the multiplicative domain of $E_0$, which is $\ell^\infty(\Gamma)\otimes1$ (the multiplicative domain of a faithful conditional expectation onto a C$^*$-algebra is exactly the range domain). By (\ref{E_0PP=P}), we have 
\begin{equation}\label{P omega P=P delta}
\CP_\omega\circ\CP|_{C(X)}=E_0\circ\CP_\omega\circ\CP|_{C(X)}=P_{\delta_x}\otimes 1:C(X)\to \ell^\infty(\Gamma)\otimes 1,
\end{equation}
 which is a $*$-homomorphism. Now we have $\CP_\omega\circ\CP|_M=\id_M=\id_{\CB_X^x} |_M$ and $\CP_\omega\circ\CP|_{C(X)}=P_{\delta_x}\otimes 1=\id_{\CB_X^x} |_{C(X)}$. By Lemma \ref{ucp on A_i to homomorphism}, we must have $\CP_\omega\circ\CP=\id_{\CB_X^x} $. Therefore, $\CP$ is completely isometric.
\\

(2) Now let us prove that $\CB_X^x $ does not depend on the choice of $x$. Fix another $y\in X$, let $\CB_X^y \subset \mathscr{A}$ be the C$^*$-subalgebra generated by $M$ and $P_{\delta_y}(C(X))\otimes 1$. Since $X$ is a $\Gamma$-boundary, there exists a net $(h_j)_J\subset \Gamma$ such that $h_j \delta_x\xrightarrow{\mathrm{w}^*} \delta_y$. Then we have $\mathrm{Ad}(\rho_{h_j})\circ P_{\delta_x}=P_{g_i\delta_x}\to P_{\delta_y}$. Take $\omega'\in \beta J\setminus J$, let 
 
 $$\CP_{\omega'}=\lim_{\omega'} \mathrm{Ad}(\rho_{h_j}\otimes u_{h_j})\in \UCP_M(\mathscr{A},\mathscr{A}).$$ 
 By the same discussion in (\ref{E_0PP=P}), we have 
 $$E_0\circ\CP_{\omega'}\circ (P_{\delta_x}\otimes 1)=\CP_{\omega'}\circ E_0\circ (P_{\delta_x}\otimes 1)=P_{\delta_y}\otimes 1:C(X)\to \ell^\infty(\Gamma)\otimes 1.$$ 
 
 Again, by Lemma \ref{ucp homomorphism}, $\CP_{\omega'}\circ (P_{\delta_x}\otimes 1)(C(X))$ is in the multiplicative domain of $E_0$, which is $\ell^\infty(\Gamma)\otimes 1$. Therefore, $\CP_{\omega'}\circ (P_{\delta_x}\otimes 1)=P_{\delta_y}\otimes 1$ and $$\CP_{\omega'}|_{P_{\delta_x}(C(X))\otimes 1}=(P_{\delta_y}\otimes 1)\circ (P_{\delta_x}\otimes 1)^{-1}: P_{\delta_x}(C(X))\otimes 1\to P_{\delta_y}(C(X))\otimes 1$$ 
 is a $*$-isomorphism. So both $P_{\delta_x}(C(X))\otimes 1$ and $M$ are contained in the multiplicative domain of $\CP_{\omega'}$. Hence $\CP_{\omega'}|_{\CB_X^x }$ is a $*$-homomorphism. Since $\CP_{\omega'}(M)=M$ and $\CP_{\omega'}(P_{\delta_x}(C(X))\otimes 1)=P_{\delta_y}(C(X))\otimes 1$, we have $\CP_{\omega'}|_{\CB_X^x} :\CB_X^x \to\CB_X^y $ is surjective. Moreover, by (1), $\CP_{\omega'}|_{\CB_X^x} \in \UCP_{M}(\CB_X^x ,\mathscr{A})$ is completely isometric. Therefore, $\CP_{\omega'}|_{\CB_X^x }:\CB_X^x \to\CB_X^y $ is an $M$-bimodular $*$-isomorphism.
\\

 (3) Assume that $(A,\tau_A)$ is amenable. Then $\mathscr{A}=B(\ell^2(\Gamma))\otimes A$ is also amenable. There exists a conditional expectation $E_{\mathscr{A}}: B(H)\to \mathscr{A}$. For any ucp map $\CP_0\in \UCP_M(\CB_X ,B(H))$, by (1), we have that $E_{\mathscr{A}}\circ \CP_0\in \UCP_M(\CB_X ,\mathscr{A})$ is completely isometric, which induces that $\CP_0$ must be completely isometric. By the arbitrariness of $\CP_0$, we know that $\CB_X $ is an $M$-boundary.
\end{proof}

From now on, we assume that $(A,\tau_A)$ is amenable. The following theorem generalizes \cite[Theorem A]{AHO23}.

\begin{theorem}\label{max amenable}
Assume that $N\in\SA(\Gamma\ltimes A)$ is an amenable $\Gamma$-invariant subalgebra with $A\subset N\subset \Gamma\ltimes A$. Then we must have 
$$N\subset \rad(\Gamma)\ltimes A .$$
In particular, when $\rad(\Gamma)=\{e\}$, $A$ is a maximal invariant amenable subalgebra.
\end{theorem}
\begin{proof}
For simplicity, we still let $M=\Gamma\ltimes A$. First, let us prove that $\ell^\infty(\Gamma/\mathrm{Rad}(\Gamma))\otimes 1\subset JN'J$. Take $x\in \FB$ and let $\CB\subset \mathscr{A}=B(\ell^2(\Gamma))\otimes A$ be the C$^*$-algebra generated by $M$ and $P_{\delta_x}(C(\FB))\otimes 1$. Then by Theorem \ref{Gamma A boundary}, $\CB$ is an $M$-boundary. 

Since $N$ is amenable, so is $JN'J\subset JA'J=\mathscr{A}$. Also note that $M\subset JN'J$. Since $\CB$ is an $M$-boundary, by Proposition \ref{exists completely isometry}, there exists a completely isometric $\CP\in \UCP_M(\CB, JN'J)$. 

Note that $\CP\in \UCP_M(\CB,\mathscr{A})$. Following the same discussion (\ref{E_0PP=P}) and (\ref{P omega P=P delta}) in the proof of Theorem \ref{Gamma A boundary}, we know that for any $y\in \CFB$, $P_{\delta_y}\otimes 1: \CFB \to \ell^\infty(\Gamma)\otimes 1\subset \mathscr{A}$ can be written as a pointwise weak limit point of $\{\mathrm{Ad}(\rho_g\otimes u_g)\circ \CP|_{\CFB}| g\in \Gamma\}$. Since $ N $ is $\Gamma$-invariant, we know that $JN'J$ is $J\Gamma J$-invariant, i.e., $\mathrm{Ad}(\rho_g\otimes u_g)$-invariant for any $g\in \Gamma$. Moreover, since $\CP(\CFB)\subset JN'J$, we have $\mathrm{Ad}(\rho_g\otimes u_g)\circ \CP(\CFB)\subset JN'J$ for any $g\in \Gamma$. Therefore, for any $y\in \FB$,
$$P_{\delta_y}(\CFB)\otimes 1\subset \overline{\mathrm{span}\{\mathrm{Ad}(\rho_g\otimes u_g)\circ \CP(\CFB)|g\in \Gamma\}}^{\mathrm{wo}}\subset JN'J.$$
Note that the von Neumann algebra generated by $P_{\delta_y}(\CFB)$ is $\ell^\infty(\Gamma/\mathrm{Stab}(y))\subset \ell^\infty(\Gamma)$. Since $\rad(\Gamma)=\mathrm{Ker}(\Gamma\car\partial_F\Gamma)=\cap_{y\in \FB}\mathrm{Stab}(y)$ \cite[Proposition 7]{Fur03}, the von Neumann algebra generated by $\{\ell^\infty(\Gamma/\mathrm{Stab}(y))|y\in \FB\}$ is 
$$\ell^\infty(\Gamma/\cap_{y\in \FB}\mathrm{Stab}(y))=\ell^\infty(\Gamma/\mathrm{Rad}(\Gamma))\subset \ell^\infty(\Gamma).$$
Therefore, we have $\ell^\infty(\Gamma/\mathrm{Rad}(\Gamma))\otimes1\subset JN'J$.

Now we have $\lag M, \ell^\infty(\Gamma/\mathrm{Rad}(\Gamma))\otimes 1\rag \subset JN'J $. Therefore, 
$$N\subset (J\lag M, \ell^\infty(\Gamma/\mathrm{Rad}(\Gamma))\otimes 1\rag J)'.$$

By taking $\Lambda=\mathrm{Rad}(\Gamma)$ in Lemma \ref{crossed product as basic construction}, we have
$$N\subset(J\lag M , \ell^\infty(\Gamma/\mathrm{Rad}(\Gamma))\otimes 1\rag J)'= \mathrm{Rad}(\Gamma)\ltimes A.$$
\end{proof}

As shown in Proposition \ref{inj weakly amenable}, injectivity (amenability) is generally stronger than weakly Zimmer amenability. However, for a $J\Gamma J$-invariant intermediate subalgebra $\mathscr{N}$ of the inclusion $M \subset \mathscr{A}$, these two properties are equivalent:
\begin{corollary}\label{weak amenable = amenable}
Following the notations above, let $H=\ell^2(\Gamma)\otimes L^2(A)$, $M=\Gamma\ltimes A$ and $\mathscr{A}=B(\ell^2(\Gamma))\otimes A$. For $\mathscr{N}\in\mathrm{vN}(H)$ with $M\subset\mathscr{N}\subset\mathscr{A}$, assume that $\mathscr{N}$ is $J\Gamma J$-invariant and $M$-weakly Zimmer amenable. Then $\mathscr{N}$ is necessarily amenable.
\end{corollary}
\begin{proof}
Follow the notations in the proof of Theorem \ref{max amenable}. By Proposition \ref{weakly amenable and FB}, we have $\partial_F M\subset\mathscr{N}$ as a sub-operator system. Since $\CB=\mathrm{C}^*(M,C(\partial_F\Gamma))$ is an $M$-boundary, by viewing $\CB$ as a subalgebra of $\partial_F M$, there exists a completely isometric $\CP\in\UCP_M(\CB,\mathscr{N})$. Then by the same discussion in the proof of Theorem \ref{max amenable}, with $JN'J$ replaced by $\mathscr{N}$, we can show that 
$$J\mathscr{N}J'\subset\rad(\Gamma)\ltimes A.$$
Hence $J\mathscr{N}J'$ is amenable and so is $\mathscr{N}$.
\end{proof}

Let $M=\Gamma\ltimes A$ and
$$\SAAM(M)=\{N\in\SA(M)\mid \mbox{$N$ is amenable}\}.$$
Then by \cite[Proposition 4.7]{AHO23}, $\SAAM(M)$ is a closed subset of $\SA(M)$. Since $\SA(M)$ is a closed subset of $\mathrm{vN}(H)$, $\SAAM(M)$ is also a closed subset of $\mathrm{vN}(H)$. 

Let 
$$\SA(A,M)=\{N\in\SA(M)\mid A\subset N\subset M\},$$
$$\SAAM(A,M)=\{N\in\SA(A,M)\mid \mbox{$N$ is amenable}\}.$$
Then $\SAAM(A,M)=\SA(A,M)\cap\SAAM(M)\subset\mathrm{vN}(H)$ is also closed and $\Gamma$-invariant since both $\SAAM(M)$ and $\SA(A,M)$ are.

\begin{definition}\label{def IRA}
Following \cite[Definition 1.1]{AHO23}, an \textbf{invariant random subalgebra} (IRA) of $\Gamma\ltimes A$ is a $\Gamma$-invariant Borel probability measure $\mu\in\prob(\SA(\Gamma\ltimes A))$ under the $\Gamma$-conjugation action. An \textbf{invariant random intermediate subalgebra} (IRIA) of $\Gamma\ltimes A$ is an IRA $\mu$ with $\mu$-a.e. $N\in\SA(\Gamma\ltimes A)$ satisfying $A\subset N\subset \Gamma\ltimes A$. An IRIA $\mu$ is \textbf{amenable} if $\mu$-a.e. $N\in\SA(\Gamma\ltimes A)$ is amenable.
\end{definition}

Let $\theta: \mathrm{vN}(H)\to\mathrm{vN}(H)$ be $\theta(N)=JN'J$. Then by \cite[Corollary 3.6]{HW98}, $\theta$ is a homeomorphism. Let
$$ \SA(M,\mathscr{A})=\{\mathscr{N}\in\mathrm{vN}(H)\mid M\subset \mathscr{N}\subset\mathscr{A}\},$$
$$ \SAAM(M,\mathscr{A})=\{\mathscr{N}\in\SA(M,\mathscr{A})\mid \mbox{$\mathscr{N}$ is amenable}\}.$$
Since $N\in\SA(A,M)$ is amenable iff $JN'J\in\SA(M,\mathscr{A})$ is, we have $ \SAAM(M,\mathscr{A})=\theta(\SAAM(A,M))\subset\mathrm{vN}(H)$ is closed. Also, both $\SA(M,\mathscr{A})$ and $\SAAM(M,\mathscr{A})$ are $\Gamma$-invariant under the action
\begin{equation}\label{g on EA}
g\mathscr{N}=J(\lambda_g\otimes 1)J\mathscr{N}J(\lambda_g^*\otimes 1)J=(\rho_g\otimes u_g)\mathscr{N}(\rho_g^*\otimes u_g^*)\ (g\in\Gamma).
\end{equation}
And $\theta:\SAAM(A,M)\to \SAAM(M,\mathscr{A})$ is $\Gamma$-equivariant.

Fix a metrizable $\Gamma$-boundary $X$ (note that the Furstenberg boundary $\FB$ may not be metrizable \cite[Corollary 3.17]{KK14}). Let $\CB_X=\mathrm{C}^*(M,C(X))$ be the $M$-boundary as in Theorem \ref{Gamma A boundary}. Let $\mathfrak{C}(\prob(X))$ be the space of all compact convex (non-empty) subsets of $\prob(X)\cong\mathrm{State}(C(X))$. Following \cite[Subsection 4.3]{AHO23} and \cite[Section 2]{BDLW16} (where the metrizability of $X$ is necessary), take a countable dense subset $\{f_n\}$ of the unit ball of $C(X)$  and define a map $f:\mathfrak{C}(\prob(X))\to \prod_n [0,1]$ by
$$f(\mathcal{C}_0)=(f_n^+(\mathcal{C}_0)) \ (\mathcal{C}_0\in\mathfrak{C}(\prob(X))),$$
where
$$f_n^+(\mathcal{C}_0)=\sup_{\nu\in \mathcal{C}_0}\mathrm{Re}
\,\nu(f_n).$$
By the Hahn-Banach separation theorem, $f$ is injective. A topology on $\mathfrak{C}(\prob(X))$ can be induced from the product topology on $\prod_n [0,1]$. Moreover, $f$ preserves the convex structures as described in \cite[Section 2]{BDLW16}: for $\mathcal{C}_1,\mathcal{C}_2\in\mathfrak{C}(\prob(X))$ and $0<\lambda<1$,
$$\lambda\mathcal{C}_1+(1-\lambda)\mathcal{C}_2=\{\lambda\nu_1+(1-\lambda)\nu_2\mid\nu_i\in\mathcal{C}_i\ (i=1,2)\}\in\mathfrak{C}(\prob(X)).$$

Define a map $\mathcal{C}: \SAAM(M,\mathscr{A})\to\mathfrak{C}(\prob(X))$ by
$$\mathcal{C}(\mathscr{N})=\{\nu\in\prob(X)\mid \mbox{$\exists \CP\in\UCP_M(\CB_X, \mathscr{N})$ with $E_0\circ\CP|_{C(X)}=P_\nu\otimes1$}\}$$
for $\mathscr{N}\in \SAAM(M,\mathscr{A})$.

\begin{proposition}\label{semi continuous}
The map $\mathcal{C}$ is well-defined, $\Gamma$-equivariant and upper semi-continuous.
\end{proposition}
\begin{proof}
For any $\mathscr{N}\in \SAAM(M,\mathscr{A})$, by Proposition \ref{exists completely isometry}, $\UCP_M(\CB_X, \mathscr{N})$ is non-empty and convex. Hence $\mathcal{C}(\mathscr{N})$ is also non-empty and convex. To prove $\mathcal{C}(\mathscr{N})$ is compact, only need to prove it is closed. For any sequence $(\nu_n)\subset\mathcal{C}(\mathscr{N})$ with $\nu_n\xrightarrow{\mathrm{w}^*}\nu_0\in\prob(X)$, take $\CP_n\in\UCP_M(\CB_X, \mathscr{N})$ with $E_0\circ\CP_n|_{C(X)}=P_{\nu_n}\otimes1$. Take $\omega\in\beta\N\setminus\N$, let $\CP_0=\lim_\omega \CP_n\in \UCP_M(\CB_X, \mathscr{N})$. Since $E_0$ is normal, we have 
$$E_0\circ\CP_0|_{C(X)}=\lim_\omega E_0\circ\CP_n|_{C(X)}=\lim_\omega P_{\nu_n}\otimes1=P_{\nu_0}\otimes1.$$
Hence $\nu_0\in\mathcal{C}(\mathscr{N})$ and $\mathcal{C}(\mathscr{N})$ is a compact convex subset of $\prob(X)$. Therefore, $\mathcal{C}: \SAAM(M,\mathscr{A})\to\mathfrak{C}(\prob(X))$ is well-defined.
\\

For any $\mathscr{N}\in \SAAM(M,\mathscr{A})$ and $g\in\Gamma$, we have
$$\UCP_M(\CB_X, g\mathscr{N})=\AD(\rho_g\otimes u_g)\circ\UCP_M(\CB_X, \mathscr{N}).$$
For any $\nu\in\mathcal{C}(\mathscr{N})$, take $\CP\in\UCP_M(\CB_X, \mathscr{N})$ with $E_0\circ\CP|_{C(X)}=P_{\nu}\otimes1$. Then $\CP'=\AD(\rho_g\otimes u_g)\circ\CP\in\UCP_M(\CB_X, g\mathscr{N} )$ satisfies 
$$\begin{aligned}
E_0\circ\CP'|_{C(X)}=&E_0\circ\AD(\rho_g\otimes u_g)\circ\CP|_{C(X)}\\
=&\AD(\rho_g\otimes u_g)\circ E_0\circ\CP|_{C(X)}\\
=&\AD(\rho_g\otimes u_g)\circ (P_{\nu}\otimes1)\\
=&P_{g\nu}\otimes1.    
\end{aligned}$$
Hence $g\nu\in \mathcal{C}(g\mathscr{N})$ for any $\nu\in\mathcal{C}(\mathscr{N})$. Therefore, $g\mathcal{C}(\mathscr{N})\subset\mathcal{C}(g\mathscr{N})$. For the same reason, we have $g^{-1}\mathcal{C}(g\mathscr{N})\subset\mathcal{C}(g^{-1}(g\mathscr{N}))$, i.e., $\mathcal{C}(g\mathscr{N})\subset g\mathcal{C}(\mathscr{N})$. Therefore, we have $g\mathcal{C}(\mathscr{N})=\mathcal{C}(g\mathscr{N})$ and $\mathcal{C}$ is $\Gamma$-equivariant.
\\

Now let us prove the upper semi-continuity of $\mathcal{C}$. Assume that $\TN_n\to\TN$ within $ \SAAM(M,\mathscr{A})$. Fix a $n_0\in\N$, only need to prove 
$$f_{n_0}^+(\mathcal{C}(\TN)) \geq \limsup_n f_{n_0}^+(\mathcal{C}(\TN_n)).$$
For $n\geq 1$, since $\mathcal{C}(\mathscr{N}_n)$ is compact, there exists a $\nu_n\in\mathcal{C}(\mathscr{N}_n)$ with $f_{n_0}^+(\mathcal{C}(\TN_n))=\mathrm{Re}\,\nu_n(f_{n_0})$. Take $\CP_n\in\UCP_M(\CB_X, \mathscr{N}_n )$ with $E_0\circ\CP_n|_{C(X)}=P_{\nu_n}\otimes1$. Take a subsequence $(f_{n_0}^+(\mathcal{C}(\TN_{n_k})))$ of $(f_{n_0}^+(\mathcal{C}(\TN_n)))$ with $$\limsup_n f_{n_0}^+(\mathcal{C}(\TN_n))=\lim_k f_{n_0}^+(\mathcal{C}(\TN_{n_k})).$$
Take $\omega\in\beta\N\setminus\N$. Let $\CP=\lim_{k\to\omega}\CP_{n_k}\in \UCP_M(\CB_X, \mathscr{A})$. Then for any $T\in \CB_X$, 
$$\CP(T)=\mathrm{wo}-\lim_{k\to\omega} \CP_{n_k}(T)\in\limsup_n \TN_n=\TN.$$ 
Hence $\CP\in\UCP_M(\CB_X, \mathscr{N} )$. Let $\nu\in\mathcal{C}(\mathscr{N})$ be the measure associated with $\CP$. Then we have $\nu=\mathrm{w}^*-\lim_{k\to\omega}\nu_{n_k}$ and
$$f_{n_0}^+(\mathcal{C}(\TN))\geq\mathrm{Re}\,\nu(f_{n_0})=\lim_{k\to\omega}\mathrm{Re}\,\nu_{n_k}(f_{n_0})=\lim_k f_{n_0}^+(\mathcal{C}(\TN_{n_k}))=\limsup_n f_{n_0}^+(\mathcal{C}(\TN_n)).$$
Therefore, $\mathcal{C}$ is upper semi-continuous.
\end{proof}

The following theorem generalizes \cite[Theorem B]{AHO23}.

\begin{theorem}\label{amenable IRS for crossed products}
Let $\mu\in\prob(\SA(M))$ be an amenable IRIA. Then for $\mu$-a.e. $N\in\SA(M)$, we have
$$N\subset \rad(\Gamma)\ltimes A. $$
In particular, when $\rad(\Gamma)=\{e\}$, we must have $\mu=\delta_A$.
\end{theorem}
\begin{proof}
Since $\mu\in\prob(\SA(M))$ is an amenable IRIA, we have $\supp \mu\subset \SAAM(A,M)$.

Fix a metrizable $\Gamma$-boundary $X$ and define $\CB_X$ as in Proposition \ref{semi continuous}. Consider the map 
$$\Phi:=\mathcal{C}\circ\theta:\SAAM(A,M)\to \SAAM(M,\mathscr{A})\to\mathfrak{C}(\prob( X)).$$
We have that $\Phi$ is $\Gamma$-equivariant since both $\mathcal{C}$ and $\theta$ are. Since $\theta$ is continuous and $\mathcal{C}$ is upper semi-continuous (Proposition \ref{semi continuous}), $\Phi$ is measurable. Hence we have $\Phi_*\mu\in\prob(\mathfrak{C}(\prob( X)))$. 

Since $\mu$ is $\Gamma$-invariant and $\Phi$ is $\Gamma$-equivariant, $\Phi_*\mu$ is also $\Gamma$-invariant. Moreover, since $ X$ is a $\Gamma$-boundary, $\prob( X)$ has no $\Gamma$-invariant closed convex proper subspace. By \cite[Lemma 2.3]{BDLW16}, we must have $\Phi_*\mu=\delta_{\prob( X)}$. Then 
$$V_X:=\{N\in\SAAM(A,M)\mid \Phi (N)=\prob( X)\}\subset \SA(M)$$
is a $\mu$-conull subset.

For any $N\in V_X$, we have $\Phi (N)=\mathcal{C}(JN'J)=\prob( X)$. Hence for any $y\in X$, there exists a $\CP\in\UCP_M(\CB_X,JN'J)$ with $E_0\circ \CP|_{ C(X)}=P_{\delta_y}\otimes 1$. Hence $E_0\circ \CP|_{ C(X)}$ is a $*$-homomorphism. By Lemma \ref{ucp homomorphism}, $\CP({ C(X)})$ is contained in the multiplicative domain of $E_0$, i.e., $\ell^\infty(\Gamma)\otimes 1$. Hence 
$$\CP|_{ C(X)}=E_0\circ \CP|_{ C(X)}=P_{\delta_y}\otimes 1: C(X)\to \ell^\infty(\Gamma)\otimes 1.$$
Therefore, $P_{\delta_y}( C(X))\otimes 1=\CP( C(X))\subset JN'J$ for any $y\in X$. 

Let $\Gamma_X:=\mathrm{Ker}(\Gamma\car X)=\cap_{y\in X}\mathrm{Stab}(y)$. Then the von Neumann algebra generated by $\{\ell^\infty(\Gamma/\mathrm{Stab}(y))|y\in X\}$ is 
$$\ell^\infty(\Gamma/\cap_{y\in X}\mathrm{Stab}(y))=\ell^\infty(\Gamma/\Gamma_X)\subset \ell^\infty(\Gamma).$$
Therefore, for any $N\in V_X$, we have $\ell^\infty(\Gamma/\Gamma_X)\otimes1\subset JN'J$.

Denote by $\mathrm{Bnd_m}(\Gamma)$ the set of metrizable $\Gamma$-boundaries. Following \cite[Proposition 7]{Fur03}, we have $\rad (\Gamma)=\cap_{X\in \mathrm{Bnd_m}(\Gamma)}\Gamma_X$. Hence for any $g\in \Gamma\setminus\rad(\Gamma)$, there exists $X_g\in\mathrm{Bnd_m}(\Gamma)$ with $g\not\in\Gamma_{X_g}$. Therefore, we have 
$$\rad (\Gamma)=\cap_{g\in \Gamma\setminus\rad(\Gamma)}\Gamma_{X_g}.$$
Let $V=\cap_{g\in \Gamma\setminus\rad(\Gamma)}V_{X_g}$. Since $\Gamma\setminus\rad(\Gamma)$ is countable, $V$ is still a $\mu$-conull subset of $\SA(M)$. And for any $N\in V$, we have that $\ell^\infty(\Gamma/\Gamma_{X_g})\otimes1\subset JN'J$ for any $g\in \Gamma\setminus\rad(\Gamma)$. And we further have
$$\ell^\infty(\Gamma/\rad(\Gamma))\otimes1=\lag\ell^\infty(\Gamma/\Gamma_{X_g})\otimes 1\rag_{g\in \Gamma\setminus\rad(\Gamma)}\subset JN'J.$$
Following the same discussion at the end of the proof of Theorem \ref{max amenable}, we have 
$$\mbox{$\ell^\infty(\Gamma/\rad(\Gamma))\otimes 1\subset JN'J$}\Rightarrow\mbox{$N\subset \rad(\Gamma)\ltimes A $}.$$
Therefore, for $\mu$-a.e. $N\in \SA(M)$, we have $N\subset \rad(\Gamma)\ltimes A $.
\end{proof}

\section{Amenable invariant random subequivalence relations of free pmp actions}\label{IRR}
In this section, we consider the case that $(A,\tau_A)=L^\infty(X,\nu_X)$ for a free probability measure preserving (pmp) action $\Gamma\car(X,\nu_X)$. For details regarding the theory of countable equivalence relations in this section, we refer to \cite[Subsection 1.5]{AP17} and \cite[Subsection 4.8]{KL16}. 

Let $\CR\subset X\times X$ be a countable nonsingular equivalence relation on a standard probability space $(X,\nu_X)$. Let $\nu$ and $\nu'$ be the canonical left and right measures on $\CR$. That is, for any measurable subset $E\subset\CR$, let
$$E^x=\{(x,y)\in E\} \mbox{ and }E_x=\{(y,x)\in E\}\ (x\in X).$$
Then we have
$$\nu(E)=\int_X|E^x|\d \nu_X(x),\ \nu'(E)= \int_X|E_x|\d \nu_X(x)\mbox{ and }\nu\sim\nu'.$$

Define an equivalence relation on the collection of subequivalence relations of $\CR$: For subequivalence relations $\CS_1,\CS_2\subset \CR$, $\CS_1\sim\CS_2$ if $\nu(\CS_1\Delta\CS_2)=0$. Let
$$\SR(\CR)=\{\mbox{subequivalence relations of $\CR$}\}/\sim.$$
For a subequivalence relation $\CS\subset\CR$, denoted by $\Bar{\CS}\in\SR(\CR)$ the image of $\CS$ under the quotient map. For $\Bar{\CS}_1,\Bar{\CS}_2\in \SR(\CR)$, we simply say that $\Bar{\CS}_1\subset\Bar{\CS}_2$ if $\nu(\CS_1\setminus\CS_2)=0$.

\begin{definition}\label{metric}
Let
$$\mathcal{E}_\mathrm{fin}=\{E\subset\CR\mbox{ measurable }\mid \nu(E)<+\infty\}.$$
For $E\in\mathcal{E}_\mathrm{fin}$, define the finite measure $\nu_E$ by
$$\nu_E(K)=\nu(E\cap K)\ (K\subset\CR\mbox{ measurable}).$$
Define the \textbf{Fr\'{e}chet-Nikodym pseudometric} associated with $\nu_E$ on $\SR(\CR)$ by 
$$d_E(\Bar{\CS}_1,\Bar{\CS}_2)=\nu_E(\CS_1\Delta\CS_2)\ (\Bar{\CS}_1,\Bar{\CS}_2\in \SR(\CR)).$$
And we endow $\SR(\CR)$ with the topology induced by the pseudometric family $\{d_E\mid E\in\mathcal{E}_\mathrm{fin}\}$, i.e., $\Bar{\CS}_n\to \Bar{\CS}$ within $\SA(L(\CR))$ iff $d_E(\Bar{\CS}_n,\Bar{\CS})\to 0$ for any $E\in\mathcal{E}_\mathrm{fin}$.
\end{definition}
Since $\nu\sim\nu'$, the topology defined above does not depend on the choice between $\nu$ and $\nu'$. 

Note that the topology defined above is equivalent to the metric topology defined in \cite[Section 2]{LM24}, and further equivalent to the topology defined in \cite[Section 4.2]{Kec17} and \cite[Discussion before Lemma 3.18]{AFH24} (for the pmp case) by \cite[Corollary 3.5]{LM24}. It is shown in \cite[Theorem 4.13]{Kec17} and \cite[Theorem A]{LM24} that this topology is Polish, hence induces a standard Borel structure.

Let us recall the definition of von Neumann algebras of equivalence relations. Following \cite[Subsection 1.5.2]{AP17}, for a sub-relation $\CS\subset\CR$, let $\mathcal{M}_b(\CS)$ be the set of all bounded Borel functions $F:\CS\to\C$ satisfying that there exists a constant $C>0$ such that for any $(x,y)\in\CS$,
$$|\{z\in X\mid F(z,y)\not=0\}|\leq C,\mbox{ and }|\{z\in X\mid F(x,z)\not=0\}|\leq C.$$
For $F\in\mathcal{M}_b(\CS)$, define the operator $L_F\in B(L^2(\CR,\nu))$ by
$$L_F(\xi)(x,y)=(F\ast\xi)(x, y)=\sum_{z\CS x}F(x,z)\xi(z, y)\ (\xi\in L^2(\CR,\nu)).$$
Then $L(\CS)\subset B(L^2(\CR,\nu))$ is the von Neumann algebra generated by $\{L_F\in B(L^2(\CR,\nu))\mid F\in\mathcal{M}_b(\CS) \}$.

Let $$\SA(L^\infty(X),L(\CR))=\{N\in\SA(L(\CR))\mid L^\infty(X)\subset N\subset L(\CR)\}$$ be the collection of intermediate subalgebras of $L^\infty(X)\subset L(\CR)$, which is a closed subset of $\SA(L(\CR))$.
\begin{theorem}\label{L(S) continuous}
The map $L:\Bar{\CS}\in\SR(\CR)\mapsto L(\CS)\in \SA(L^\infty(X),L(\CR))$ is a homeomorphism.
\end{theorem}
\begin{proof}
Firstly, by \cite[Theorem 1]{FM77} and \cite[Lemma 3.3 and Corollary 3.5]{Aoi03}, we know that $L:\SR(\CR)\to \SA(L^\infty(X),L(\CR))$ is a well-defined bijection.

Let $\phi\in L(\CR)_*$ be the normal faithful state associated with the measure $\nu$, i.e., $\phi=\lag\,\cdot\,\mathds{1}_X,\mathds{1}_X\rag_{L^2(\CR,\nu)}$. Then we have $L^2(\CR,\nu)=L^2(L(\CR),\phi)$ with cyclic vector $\xi_\phi=\mathds{1}_X$. Let $J=J_\phi$ be the modular conjugation operator of $(L(\CR),\phi)$. Then for $x_n,x\in L(\CR)$, we have
\begin{equation}\label{xn to x so}
\begin{aligned}
    &x_n\xrightarrow{\mathrm{so}} x\\
\iff& (x_n-x) J y J \xi_\phi\to 0, \ \forall y\in L(\CR) &(\mbox{$\overline{JL(\CR)J\xi_\phi}=L^2(\CR,\nu)$})\\
\iff& J y J (x_n-x) \xi_\phi\to 0, \ \forall y\in L(\CR))\\
\iff& (x_n-x) \xi_\phi\to 0.
\end{aligned}  
\end{equation}

Let $\{\sigma_t^\phi\}_{t\in\R}$ be the modular automorphism group of $(L(\CR),\phi)$ and
$$\SA_\phi(L^\infty(X),L(\CR))=\{N\in\SA(L^\infty(X),L(\CR))\mid \sigma^\phi_t(N)=N,\ \forall t\in\R\}.$$
For any $N\in\SA(L^\infty(X),L(\CR))$, by \cite[Theorem 1.1]{Aoi03}, there exists a normal faithful conditional expectation $E_N:L(\CR)\to N$. Moreover, since $N'\cap L(\CR)\subset L^\infty(X)'\cap L(\CR)=L^\infty(X)\subset N$, by \cite[Proposition 4.3 and Theorem 4.2]{TakII}, we have $\phi=\phi\circ E_N$ and $N\in \SA_\phi(L^\infty(X),L(\CR))$. Therefore, we have
\begin{equation}\label{SA=SA_phi}
\SA(L^\infty(X),L(\CR))=\SA_\phi(L^\infty(X),L(\CR)).
\end{equation}

For $\Bar{\CS}_n, \Bar{\CS}\in \SR(\CR)$, we have
\begin{equation}\label{continuity of L-1}
\begin{aligned}
    &L(\CS_n)\to L(\CS)\\
\iff& E_{L(\CS_n)}(x)\xrightarrow{\mathrm{so}^*} E_{L(\CS)}(x), \ \forall x\in L(\CR) &\hspace{-10em}(\mbox{by (\ref{SA=SA_phi}) and \cite[Remark 2.11]{HW98}})\\
\iff& (E_{L(\CS_n)}(x)-E_{L(\CS)}(x)) \xi_\phi\to 0, \ \forall x\in L(\CR) &(\mbox{by (\ref{xn to x so})})\\
\iff& e_{L(\CS_n)}(x\xi_\phi)-e_{L(\CS)}(x\xi_\phi)\to 0, \ \forall x\in L(\CR)&(\mbox{$E_{L(\CS)}(x)\xi_\phi=e_{L(\CS)}(x\xi_\phi)$})\\
\iff& e_{L(\CS_n)}\xrightarrow{\mathrm{so}} e_{L(\CS)}\\
\iff& \mathds{1}_{\CS_n}\xrightarrow{\mathrm{so}}\mathds{1}_\CS & (e_{L(\CS)}=\mathds{1}_\CS)\\
\iff& \nu_E(\CS_n\Delta\CS)=\Vert (\mathds{1}_{\CS_n}-\mathds{1}_\CS)\mathds{1}_E\Vert^2\to 0,\ \forall E\in\mathcal{E}_\mathrm{fin} \vspace{-10em}& (L^2(\CR,\nu)=\overline{\mathrm{span}}\{\mathds{1}_E\}_{E\in\mathcal{E}_\mathrm{fin}})\\
\iff& \Bar{\CS}_n\to \Bar{\CS}.
\end{aligned}  
\end{equation}
Therefore, the bijection $$L:\SR(\CR)\to \SA(L^\infty(X),L(\CR))$$ is a homeomorphism.

\end{proof}

The study of the space of subequivalence relations was initiated by Kechris and has yielded many remarkable results. Moreover, Theorem \ref{L(S) continuous} can offer an operator algebraic approach to some of these classical results. For example, the following corollary generalizes and gives a new proof to \cite[Theorem 8.1 and 10.1]{Kec17}, where the pmp case was considered.
\begin{corollary}\label{am and erg}
Let
$$\SR_{\mathrm{am}}(\CR)=\{\Bar{\CS}\in\SR(\CR)\mid \CS\mbox{ is amenable}\},$$
and
$$\SR_{\mathrm{erg}}(\CR)=\{\Bar{\CS}\in\SR(\CR)\mid \CS \mbox{ is ergodic}\}.$$
Then $\SR_{\mathrm{am}}(\CR)$ and $\SR_{\mathrm{erg}}(\CR)$ are G$_\delta$ subsets of $\SR(\CR)$. Moreover, when $\CR$ is pmp, $\SR_{\mathrm{am}}(\CR)$ is a closed subset.
\end{corollary}
\begin{proof}
Let $\SAAM(L(\CR))$ (resp. $\mathrm{SF}(L(\CR))$) be the collection of amenable subalgebras (resp. subfactors) of $\SA(L(\CR))$. Since a subequivalence relation $\CS\subset \CR$ is amenable (resp. ergodic) if and only if $L(\CS)$ is amenable (resp. a factor), we have 
$$\SR_{\mathrm{am}}(\CR)=L^{-1}(\SAAM(L(\CR)))\mbox{ and }\SR_{\mathrm{erg}}(\CR)=L^{-1}(\mathrm{SF}(L(\CR))).$$
Also note that by Theorem \ref{L(S) continuous}, the map $L: \Bar{\CS}\in\SR(\CR)\mapsto L(\CS)\in\SA(L(\CR))$ is continuous and by \cite[Theorem 5.2 and Corollary 3.11]{HW98}, $\SAAM(L(\CR)),\mathrm{SF}(L(\CR))\subset\SA(L(\CR))$ are G$_\delta$ subsets. Therefore, $\SR_{\mathrm{am}}(\CR)$ and $\SR_{\mathrm{erg}}(\CR)$ are also G$_\delta$ subsets of $\SR(\CR)$.

When $\CR$ is pmp, i.e., $L(\CR)$ is tracial, it follows from \cite[Proposition 4.7]{AHO23} that $\SAAM(L(\CR))\subset\SA(L(\CR))$ is a closed subset. Hence $$\SR_{\mathrm{am}}(\CR)=L^{-1}(\SAAM(L(\CR)))$$ is also a closed subset of $\SR(\CR)$.
\end{proof}

\begin{remark} 
Recall that a countable measured groupoid $(\mathcal{G},\nu_s,\nu_t)$ is a Borel groupoid satisfying that $\mathcal{G}^{(0)}=(X,\nu_X)$ is a standard probability space, and $s^{-1}(x)$ and $t^{-1}(x)$ are at most countable for any $x\in X$. And the measure $\nu_s$ and $\nu_t$ on $\mathcal{G}$ satisfies that for any $E\subset\mathcal{G}$ measurable,
$$\nu_s(E)=\int_X |E\cap s^{-1}(x) |\d \nu_X(x), \ \nu_t(E)=\int_X |E\cap t^{-1}(x) |\d \nu_X(x)\mbox{ and }\nu_s\sim\nu_t.$$

In Definition \ref{metric}, we can replace the countable nonsingular equivalence relation $(\CR,\nu,\nu')$ with a countable measured groupoid $(\mathcal{G},\nu_s,\nu_t)$. Then following the same proof in \cite[Theorem 2.6]{LM24}, we also define a Polish topology on the space of subgroupoids of $\mathcal{G}$ (modulo null sets). 

Moreover, for a wide subgroupoid $\mathcal{H}$ (i.e., $\mathcal{H}\leq\mathcal{G}$ with $\mathcal{G}^{(0)}\subset\mathcal{H}$), a natural conditional expectation $E_{L(\mathcal{H})}:L(\mathcal{G})\to L(\mathcal{H})$ can be given by $L_F\mapsto L_{\mathds{1}_\mathcal{H}\cdot F}$ ($F\in\mathcal{M}_b(\mathcal{G})$), which preserves the normal faithful state associated with $\nu_s$. Following the same discussion (\ref{continuity of L-1}) in the proof of Theorem \ref{L(S) continuous}, we can show that the map from wide subgroupoids of $\mathcal{G}$ to subalgebras of $L(\mathcal{G})$ is a topological embedding. 

Just like Corollary \ref{am and erg}, we also have the notion of G$_\delta$-ness of the subspaces of amenable wide subgroupoids and ergodic ICC wide subgroupoids (for ICC groupoids and the factoriality of $L(\mathcal{G})$, see \cite{BCDK2024}), as well as the closeness of the space of amenable wide subgroupoids in the pmp case.
\end{remark}

From now on, we consider $\CR=\CR_{\Gamma\car X}$ for a free pmp action $\Gamma \car (X,\nu_X)$. Then $\SR(\CR_{\Gamma\car X})$ admits a natural continuous $\Gamma$-action given by 
\begin{equation}\label{action on SR}
g\Bar{\CS}=\overline{g\CS}:=\overline{\{(gx,gy)\in X\times X\mid (x, y)\in \CS\}}
\end{equation}
for $g\in\Gamma$ and $\Bar{\CS}\in\SR(\CR_{\Gamma\car X})$.

With the group action defined above, we have the following definition as an analogue of invariant random subgroup/subalgebra. In contrast, a similar definition for the equivalence relations on countable discrete groups appears in \cite[Chapter 15]{Kec17}.
\begin{definition}\label{def IRR}
An \textbf{invariant random subequivalence relation} (abbreviated as IRR) of $\CR_{\Gamma\car X}$ is a $\Gamma$-invariant Borel probability measure on $\SR(\CR_{\Gamma\car X})$. An IRR $\mu\in\prob(\CR_{\Gamma\car X})$ is \textbf{amenable} if $\mu$-a.e. $\Bar{\CS}\in\CR_{\Gamma\car X}$ is amenable.
\end{definition}

As a direct corollary of Theorem \ref{amenable IRS for crossed products}, we have the following theorem regarding amenable IRRs.
\begin{theorem}\label{Thm IRR}
For a free pmp action $\Gamma \car (X,\nu_X)$, let $\mu\in\prob(\SR(\CR_{\Gamma\car X}))$ be an amenable IRR. Then for $\mu$-a.e. $\Bar{\CS}\in\SR(\CR_{\Gamma\car X})$, we have
$$\Bar{\CS}\subset \Bar{\CR}_{\rad(\Gamma)\car X}.$$
In particular, any $\Gamma$-invariant amenable subequivalence relation of $\CR_{\Gamma\car X}$ must be a subequivalence relation of $\CR_{\rad(\Gamma)\car X}$, up to null sets.
\end{theorem}
\begin{proof}
Let $\CR=\CR_{\Gamma\car X}$. Note that $L(\CR)$ admits a natural $\Gamma$-action given by
$$gT=L_{S_g}TL_{S_g}^*\ (g\in \Gamma, \ T\in L(\CR)),$$
where $S_g\in \mathcal{M}_b(\CR)$ is the characteristic function of $\{(gx,x)\mid x\in X\}$. Hence $\SA(L(\CR))$ admits a natural $\Gamma$-action given by
$$gN=L_{S_g}NL_{S_g}^*\ (g\in \Gamma, \ N\in\SA(L(\CR))).$$
Moreover, for any $g\in\Gamma$ and subequivalence relation $\CS\subset\CR$, we have 
$$\mathcal{M}_b(g\CS)=\{F(g^{-1}\,\cdot\,)\mid F\in\mathcal{M}_b(\CS)\}=S_g\ast\mathcal{M}_b(\CS)\ast S_g^*.$$ 
Hence
$$\{L_F\mid F\in\mathcal{M}_b(g\CS)\}=L_{S_g}\{L_F\mid F\in\mathcal{M}_b(\CS)\}L_{S_g}^*\subset B(L^2(\CR,\nu)).$$
Therefore, we have $L(g\CS)=L_{S_g}L(\CS)L_{S_g}^*$. So the map $L:\Bar{\CS}\in\SR(\CR)\mapsto L(\CS)\in \SA(L^\infty(X),L(\CR))$ is a $\Gamma$-equivariant homeomorphism by Theorem \ref{L(S) continuous}.

Following \cite[Remark 1.5.6]{AP17}, there is a $*$-isomorphism $\theta:L(\CR)\cong L(\Gamma \car X)$ with $\theta(L_{S_g})=u_g$ for $g\in \Gamma$ and $\theta|_{L^\infty(X)}=\id_{L^\infty(X)}$. Therefore,
$$\Theta: N\in \SA(L^\infty(X),L(\CR))\mapsto\theta(N)\in\SA(L^\infty(X),L(\Gamma \car X))$$
is a $\Gamma$-equivariant homeomorphism.

Now we consider the map
$$\Phi:=\Theta\circ L: \SR(\CR)\to\SA(L^\infty(X),L(\CR))\to\SA(L^\infty(X),L(\Gamma \car X)).$$
Since both $L$ and $\Theta$ are $\Gamma$-equivariant homeomorphisms, $\Phi$ is also a $\Gamma$-equivariant homeomorphisms, hence measurable.

Assume that $\mu\in\prob(\SR(\CR))$ is a $\Gamma$-IRR with $\mu(\SR_{\mathrm{am}}(\CR))=1$. Then we have $\Phi_*\mu\in\prob(\SA(L(\Gamma \car X)))$. Since $\mu$ is $\Gamma$-invariant and $\Phi$ is $\Gamma$-equivariant, we have that $\Phi_*\mu$ is also $\Gamma$-invariant.

Take $(A,\tau_A)= L^\infty(X,\nu_X)$, $M=\Gamma \ltimes A=L(\Gamma \car X)$ and follow the notations from Section \ref{IRA}. Note that $\CS$ is amenable if and only if $\Phi(\Bar{\CS})=\theta(L(\CS))$ is amenable.
Therefore, we have $\SR_{\mathrm{am}}(\CR)=\Phi^{-1}(\SAAM(A,M))$ and
$$\Phi_*\mu(\SAAM(A,M))=\mu(\Phi^{-1}(\SAAM(A,M))=\mu(\SR_{\mathrm{am}}(\CR))=1.$$

Now $\Phi_*\mu\in\prob(\SA(L(\Gamma \car X)))$ satisfies all the conditions in Theorem \ref{amenable IRS for crossed products}. Therefore, for $\Phi_*\mu$-a.e. $N\in\SA(L(\Gamma \car X))$, we have 
$$N\subset L(\rad(\Gamma)\car X).$$
That is, for $\mu$-a.e. $\Bar{\CS}\in\SR(\CR_{\Gamma\car X})$, we have
$$L(\CS)\subset \theta^{-1}(L(\rad(\Gamma)\car X))=L(\CR_{\rad(\Gamma)\car X}),$$
which is equivalent to
$$\Bar{\CS}\subset \Bar{\CR}_{\rad(\Gamma)\car X}.$$
\end{proof}
\section{Appendix}\label{appendix}\begin{center}by \textsc{Tattwamasi Amrutam and Yongle Jiang}\end{center}

In this section, we present a significant generalization of Theorem~\ref{Thm B} and Theorem~\ref{Thm C}. While Theorem~\ref{Thm B} and Theorem~\ref{Thm C} establish crucial structural properties for amenable $\Gamma$-invariant von Neumann subalgebras, they rely on the non-commutative boundary approach, which necessitates the assumption that $(A,\tau_A)$ must be amenable and algebras must contain $A$. We demonstrate that the Furstenberg boundary $\partial_F\Gamma$, associated with the group $\Gamma$, provides sufficient analytical power to address these questions in the context of tracial crossed products. This approach allows us to remove the assumption that $(A,\tau_A)$ is amenable and $N$ contains $A$  while preserving their essential conclusions. In particular, we present alternative proofs of Theorem~\ref{Thm B} and \ref{Thm C} and generalize them.

Throughout this section, $(\mathcal{N},\tau)$ will denote a tracial von Neumann algebra on which $\Gamma$ acts by preserving the trace.
\subsection{Description of the Crossed Product}
\label{subsec:commutant}
We shall view $\mathcal{N}\rtimes\Gamma\subset \mathbb{B}(L^2(\mathcal{N},\tau)\otimes \ell^2\Gamma)$. Moreover, 
Let $X$ be a minimal $\Gamma$-space. Fixing $x_0\in X$ gives us a unital $\Gamma$-equivariant injective $*$-
homomorphism $i: C(X)\to \ell^{\infty}(\Gamma)$, given by $i(f)(t) = f(tx_0)$. We can view $\ell^{\infty}(\Gamma)$ as multiplication operators on $\mathbb{B}(\ell^2(\Gamma))$. For $f\in \ell^{\infty}(\Gamma)$, the map $M(f):\ell^2(\Gamma)\to\ell^2(\Gamma)$ defined by $M(f)(\delta_t)=f(t)\delta_t$ is linear and bounded. Therefore, we can identity $C(X)$ with its image $1\otimes M(C(X))\subset 1\otimes M(\ell^{\infty}(\Gamma))$ inside $1\otimes \mathbb{B}(\ell^2(\Gamma))$. While this embedding may not be canonical, it is enough to fix one embedding for our purposes. Under this identification, we see that every element of $C(X)$ commutes with $\mathcal{N}$.
\begin{lemma}
\label{lem:irreducibility}
Let $(\mathcal{N},\tilde{\tau})$ be a tracial von Neumann algebra. Let $\Gamma\curvearrowright (\mathcal{N},\tilde{\tau})$ be a trace preserving action. Let $\mathcal{M}$ be an amenable $\Gamma$-invariant subalgebra of $\mathcal{N}\rtimes\Gamma$. Denote by $\tau$ the trace $\tilde{\tau}\circ\mathbb{E}$. Let $$\text{Hype}_{\tau}(\mathcal{M})=\left\{\varphi:\varphi\text{ is a $\mathcal{M}$-hyperstate and }\varphi|_{\mathcal{N}\rtimes\Gamma}=\tau\right\}$$
Then,
\[\text{Hype}_{\tau}(\mathcal{M})|_{C(\partial_F\Gamma)}=\text{Prob}(\partial_F\Gamma).\]
\begin{proof}
The set $\text{Hype}_{\tau}(\mathcal{M})$ being non-empty is a consequence of \cite[Proposition~2.4]{AHO23}. The claim follows from the irreducibility of $\text{Prob}(\partial_F\Gamma)$.
\end{proof}
\end{lemma}

We now prove a general version of Theorem~\ref{Thm B}.

\begin{letterthm}
\label{prop: main prop in appendix}
Let $(\mathcal{N},\tilde{\tau})$ be a tracial von Neumann algebra. Let $\Gamma\curvearrowright (\mathcal{N},\tilde{\tau})$ be a trace preserving action. Then for any amenable $\Gamma$-invariant subalgebra $\mathcal{M}\subset\mathcal{N}\rtimes\Gamma$, we must have
$$\mathcal{M}\subset\mathcal{N}\rtimes\text{Rad}(\Gamma).$$
In particular, when $(\mathcal{N},\tilde{\tau})$ is amenable, $\mathcal{N}\rtimes\text{Rad}(\Gamma)$ is the maximal amenable $\Gamma$-invariant subalgebra of $\mathcal{N}\rtimes\Gamma$.
\begin{proof}

The proof is motivated by the proof from \cite{BC2015}.
Let $\mathcal{M}$ be an amenable $\Gamma$-invariant subalgebra of $\mathcal{N}\rtimes\Gamma$. $u\in\mathcal{M}$ be an unitary element. Let $\|\cdot\|_2$-norm denote the norm induced on $\mathcal{N}\rtimes\Gamma$ by $\tau=\tilde{\tau}\circ\mathbb{E}$. Note that here $\mathbb{E}:\mathcal{N}\rtimes\Gamma\to\mathcal{N}$ is the canonical conditional expectation. Since $\text{Rad}(\Gamma)=\text{Ker}(\Gamma\curvearrowright\partial_F\Gamma)$~\cite{BKKO14}, we shall show that $\|\mathbb{E}_{\Gamma_x}(u)\|_2=1$ for all $x\in \partial_F\Gamma$ from where the claim shall follow. Note that $\mathbb{E}_{\Gamma_x}:\mathcal{N}\rtimes\Gamma\to\mathcal{N}\rtimes\Gamma_x$ is the canonical conditional expectation. Let $\epsilon>0$. Let $u_0=\sum_{i=1}^na_i\lambda(s_i)\in\mathcal{N}\rtimes\Gamma$ be such that \begin{equation}
\label{firstapprox}
\|u^*-u_0\|_2<\epsilon.\end{equation} Let us write $F=\{s_1,s_2,\ldots,s_n\}$. Then, we can rewrite $$u_0=\sum_{s\in F\cap\Gamma_x}a_s\lambda(s)+\sum_{s\in F\cap(\Gamma_x)^c}a_s\lambda(s)$$ Now, we can find $f\in C(\partial_F\Gamma)$ such that $f(x)=1$ and $f(s^{-1}x)=0$ for all $s\in F\cap(\Gamma_x)^c$. Using Lemma~\ref{lem:irreducibility}, we can find a $\mathcal{M}$-hyperstate $\varphi$ such that $\varphi|_{\mathcal{N}\rtimes\Gamma}=\tau$ and $\varphi|_{C(\partial_F\Gamma)}=\delta_x$.
Let us now observe that
\begin{align*}
\left|\varphi\left((uu_0-1)f\right)\right|&\le \sqrt{\varphi\left((uu_0-1)(uu_0-1)^*\right)}\sqrt{\varphi(f^*f)}\\&=\|uu_0-1\|_2&\text{$\left(\varphi|_{\mathcal{N}\rtimes\Gamma}=\tau\right)$}\\&\le \|u^*-u_0\|_2<\epsilon.    
\end{align*}
Therefore,
\begin{align*}
\left|\varphi\left(u_0fu\right)\right|&= \left|\varphi\left((uu_0f\right)\right|\\&=\left|\varphi\left((uu_0-1)f\right)+\varphi\left(f\right)\right|\\&\ge \left|\varphi(f)-|\varphi\left((uu_0-1)f\right)|\right|\ge 1-\epsilon.    
\end{align*}
To reiterate, 
\begin{equation}
\label{ineq:firstinequality}
\left|\varphi\left(u_0fu\right)\right|\ge 1-\epsilon.   \end{equation}
On the other hand,
\begin{align*}
&\left|\varphi\left((u_0fu\right)\right|\\&\le\left|\varphi\left(\left(\sum_{s\in F\cap(\Gamma_x)}a_s\lambda(s)\right)fu\right)\right|+  \left|\varphi\left(\left(\sum_{s\in F\cap(\Gamma_x)^c}a_s\lambda(s)\right)fu\right)\right|\\&\le \left|\varphi\left(\mathbb{E}_{\Gamma_x}(u_0)fu\right)\right|+\sum_{s\in F\cap(\Gamma_x)^c}\left|\varphi\left(a_s\lambda(s)fu\right)\right| 
\end{align*}
Now, for every $s\in F\cap(\Gamma_x)^c$, using the Cauchy-Schwartz inequality, we see that
\begin{align*}
\left|\varphi\left(a_s\lambda(s)fu\right)\right|&\le\sqrt{\varphi\left(a_s(s.f(s.f)^*)a_s^*\right)}    
\end{align*}
Since $s.f\in C(\partial_F\Gamma)$, $a_s\in \mathcal{N}$ and every element of $C(\partial_F\Gamma)$ commutes with $\mathcal{N}$, we see that
\begin{align*}\varphi\left(a_s(s.f(s.f)^*)a_s^*\right)&=\varphi\left((s.f(s.f)^*)a_sa_s^*\right)\\&=f(s^{-1}x)\overline{f(s^{-1}x)}\varphi\left(a_sa_s^*\right)\\&=0.\end{align*}
Consequently, it follows that 
\[\sum_{s\in F\cap(\Gamma_x)^c}\left|\varphi\left(a_s\lambda(s)fu\right)\right|=0.\]
Combining this along with equation~\eqref{firstapprox}, equation~\eqref{ineq:firstinequality} and Cauchy-Schwartz inequality, we see that
\begin{align*}
1-\epsilon\le \left|\varphi(u_0fu)\right|&\le\left|\varphi\left(\mathbb{E}_{\Gamma_x}(u_0)fu\right)\right|\\&\le\left\|\mathbb{E}_{\Gamma_x}(u_0)\right\|_2\\&\le\left\|\mathbb{E}_{\Gamma_x}(u)\right\|_2+\epsilon
\end{align*}
As a result, it follows that $\left\|\mathbb{E}_{\Gamma_x}(u)\right\|_2\ge 1-2\epsilon$. Since $\epsilon>0$ is arbitrary, it follows that $u\in\mathcal{N}\rtimes\Gamma_x$. Since $x\in\partial_F\Gamma$ is arbitrary, it follows that 
\[u\in\bigcap_{x\in\partial_F\Gamma}\left(\mathcal{N}\rtimes\Gamma_x\right)=\mathcal{N}\rtimes\text{Rad}(\Gamma).\]
The proof is complete.
\end{proof}
\end{letterthm}

We present another proof of the above theorem, which follows closely the strategy of proving \cite[Lemma 3.2]{AHO23}.

\begin{proof}[2nd proof of Theorem \ref{prop: main prop in appendix}]
Let $\mathcal{M}$ be an amenable $\Gamma$-invariant subalgebra of $\mathcal{N}\rtimes \Gamma$. Fix any $x\in \mathcal{M}$. It suffices to show that $\tau(x\lambda(s)a)=0$ for all $s\in\Gamma\setminus \text{Rad}(\Gamma)$ and all $a\in \mathcal{N}$.

Indeed, write $x=\sum_{s\in\Gamma}x_s\lambda(s)$ for its Fourier expansion as an element in $\mathcal{N}\rtimes \Gamma$. The above implies that $\tau(x_sx_s^*)=\tau(x\lambda(s^{-1})x_s^*)=0$ for all $s\in \Gamma\setminus \text{Rad}(\Gamma)$. Thus $x_s=0$ for all $s\in \Gamma\setminus \text{Rad}(\Gamma)$, hence $x\in \mathcal{N}\rtimes \text{Rad}(\Gamma)$.

To show $\tau(x\lambda(s)a)=0$, we fix any $\phi\in \text{Hype}_{\tau}(\mathcal{M})$ such that $\phi|_{C(\partial_F\Gamma)}=\delta_x$, where $x\in \text{Prob}(\partial_F\Gamma)$ is any point such that $sx\neq x$. Note that the existence of $\phi$ and $x$ is guaranteed by Lemma \ref{lem:irreducibility} and $\text{Rad}(\Gamma)=\text{Ker}(\Gamma\curvearrowright \partial_F\Gamma)$ respectively. Note that $C(X)$ is contained in the multiplicative domain of $\phi$.
Take any $f\in C(\partial_F\Gamma)$ such that $0\leq f\leq 1$, $f(x)=1$ and $f(s^{-1}x)=0$. Note that $\phi(f)=f(x)=1$.

Then we have the following estimate.
\begin{align*}
&\left|\tau(x\lambda(s)a)\right|\\
&=\left|\phi(x\lambda(s)a)\right|~(\text{since $\phi|_{\mathcal{N}\rtimes \Gamma}=\tau$})\\
&=\left|\phi(x\lambda(s)a)\phi(f)\right|~(\text{since $\phi(f)=1$})\\
&=\left|\phi(x\lambda(s)af)\right|~(\text{since $f$ belongs to the multiplicative domain of $\phi$})\\
&=\left|\phi(x\sigma_s(a)\sigma_s(f)\lambda(s))\right|\\
&=\left|\phi(x\sigma_s(a)\sqrt{\sigma_s(f)}\sqrt{\sigma_s(f)}\lambda(s))\right|\\
&\leq ((\phi(x\sigma_s(a)\sigma_s(f)\sigma_s(a^*)x^*))^{\frac{1}{2}}(\phi(f))^{\frac{1}{2}}~(\mbox{by Cauchy-Schwarz inequality})\\
&=(\phi(\sigma_s(a)\sigma_s(f)\sigma_s(a^*)x^*x)^{\frac{1}{2}}\cdot 1~(\text{since $x\in \mathcal{M}$ and $\phi$ is $\mathcal{M}$-central})\\
&=(\phi(\sigma_s(a)\sqrt{\sigma_s(f)}\sqrt{\sigma_s(f)}\sigma_s(a^*)x^*x)^{\frac{1}{2}}\cdot 1\\
&\leq (\phi(\sigma_s(a)\sigma_s(f)\sigma_s(a^*)))^{\frac{1}{4}}(\phi(x^*x\sigma_s(a)\sigma_s(f)\sigma_s(a^*)x^*x))^{\frac{1}{4}}\\
&~(\text{by Cauchy-Schwarz inequality})\\
&=(\phi(\sigma_s(f)\sigma_s(a)\sigma_s(a^*)))^{\frac{1}{4}}(\phi(x^*x\sigma_s(a)\sigma_s(f)\sigma_s(a^*)x^*x))^{\frac{1}{4}}\\
&~(\text{since $\sigma_s(f)\in C(X)$ which commutes with $\mathcal{N}\ni \sigma_s(a)$})\\
&=(\phi(\sigma_s(f))\phi(\sigma_s(aa^*)))^{\frac{1}{4}}(\phi(x^*x\sigma_s(a)\sigma_s(f)\sigma_s(a^*)x^*x))^{\frac{1}{4}}\\
&=0~(\text{since $\phi(\sigma_s(f))=(\sigma_s(f))(x)=f(s^{-1}x)=0$}).
\end{align*}
Hence, $\tau(x\lambda(s)a)=0$ is proved.
\end{proof}

\begin{remark}
With the same proof of \cite[Theorem B]{AHO23}, the condition in Lemma \ref{lem:irreducibility} is satisfied for almost every $\mathcal{M}$. Then following the same proof as  Theorem~\ref{prop: main prop in appendix}, we may prove a general version of Theorem \ref{Thm C} in this paper. In particular, it follows that given an IRA $\mu$ supported on amenable subalgebras of $\mathcal{N}\rtimes\Gamma$, $\mu$-a.e $\mathcal{M}$ is contained in $\mathcal{N}\rtimes\text{Rad}(\Gamma)$.
\end{remark}

Next, we observe that  \cite[Theorem 1.4]{BC2015} can be strengthened to the following theorem. Note that the notion of singular subgroups was introduced in \cite[Definition 1.1]{BC2015}.

\begin{theorem}\label{thm: maximal inv subalgebras wrt singular subgroups}
Suppose that $\Gamma$ is a discrete countable group admitting an amenable, singular subgroup $\Lambda$. Then for any trace preserving action $\Gamma\curvearrowright (Q,\tau)$ on a finite von Neumann algebra, any $\Lambda$-invariant amenable von Neumann subalgebra inside $Q\rtimes \Gamma$ is contained in $Q\rtimes \Lambda$.
\end{theorem}
\begin{proof}
Let $A$ be any $\Lambda$-invariant von Neumann subalgebra in $M:=Q\rtimes \Gamma$. We aim to show $A\subseteq Q\rtimes \Lambda$.

The proof is the same as in \cite[Theorem 1.4]{BC2015} with only one modification. Recall that in \cite[Theorem 4.1]{BC2015}, the (stronger) assumption that $A$ contains $Q\rtimes \Lambda$ is only used to guarantee that the $A$-central state $\phi: B(L^2(Q)\bar{\otimes}\ell^2(\Gamma))\rightarrow \mathbb{C}$ with $\phi|_M=\tau$ satisfies that $\phi$ is $\Lambda$-central since $\Lambda\subset Q\rtimes \Lambda\subseteq A$. But under our weaker assumption, we can still find an $A$-central state $\tilde{\phi}$ such that $\tilde{\phi}|_M=\tau$ and meanwhile $\Lambda$-central. 

Indeed, we just need to start with an $A$-central state $\phi: B(L^2(Q)\bar{\otimes}\ell^2(\Gamma))\rightarrow \mathbb{C}$ with $\phi|_M=\tau$ and then define $\tilde{\phi}$ as any cluster point of the set of states $\{\phi_n: n\geq 1\}$, where $\phi_n(T):=\frac{\sum_{s\in F_n}\phi(\lambda(s^{-1})T\lambda(s)))}{\sharp F_n}$ for any $T\in B(L^2(Q)\bar{\otimes}\ell^2(\Gamma))$ and $\{F_n: n\geq 1\}$ is a F\o lner sequence for the amenable subgroup $\Lambda$. Then we can repeat the proof of \cite[Theorem 1.4]{BC2015} by replacing $\phi$ there with $\tilde{\phi}$ constructed here.
\end{proof}

The above theorem can be viewed as a partial, stronger version of Theorem~\ref{prop: main prop in appendix}. 

\begin{letterthm}
\label{thm:singularmaximal}
Suppose that $\Gamma$ is a discrete countable group admitting an amenable, singular subgroup. Then for any trace preserving action $\Gamma\curvearrowright (Q,\tau)$ on a finite von Neumann algebra, any $\Gamma$-invariant amenable von Neumann subalgebra inside $Q\rtimes \Gamma$ is contained in $Q\rtimes \text{Rad}(\Gamma)$.
\end{letterthm}
\begin{proof} This is a corollary of Theorem \ref{thm: maximal inv subalgebras wrt singular subgroups} once we  observe that $\text{Rad}(\Gamma)=\cap_{\Lambda\in J}\Lambda$, where $J$ is the collection of all amenable, singular subgroups in $\Gamma$.
Indeed, on the one hand, note that $\text{Rad}(\Gamma)$ is the largest normal amenable normal subgroup in $\Gamma$ and singularity property is invariant under group automorphism of $\Gamma$ and hence $\cap_{\Lambda\in J}\Lambda$ is a normal amenable subgroup in $\Gamma$. Therefore, $\text{Rad}(\Gamma)\supseteq \cap_{\Lambda\in J}\Lambda$. 
On the other hand,
$\text{Rad}(\Gamma)$ also equals the intersection of all maximal amenable subgroups in $\Gamma$ and every singular amenable subgroup is also maximal amenable by \cite[Lemma 1.2]{BC2015}, therefore, $\text{Rad}(\Gamma)\subseteq \cap_{\Lambda\in J}\Lambda$. Hence, we have proved that $\text{Rad}(\Gamma)=\cap_{\Lambda\in J}\Lambda$.  
\end{proof}

\begin{remark}
Note that not every non-amenable group $\Gamma$ admits an amenable singular subgroup. Indeed, let $\Gamma$ be a (non-amenable) Tarski monster group with the property that every proper non-trivial subgroup is a finite cyclic group, see \cite{OL79,OL80p1,OL80p2}. First note that $\Gamma$ contains no finite index subgroups and hence is I.C.C.
Then we claim that $\Gamma$ has no singular amenable subgroups. Assume not, then let  $\Lambda$ be an amenable singular subgroup. Since $\Lambda$ is maximal amenable by \cite[Lemma 1.2]{BC2015}, we deduce that $\Lambda$ is a proper non-trivial subgroup and thus $\Lambda=\frac{\mathbb{Z}}{n\mathbb{Z}}$ for some $n\geq 2$. To get a contradiction, we need to observe that $L(\Lambda)$ is not maximal amenable in $L(\Gamma)$. The reason is as follows:  we may first find $n$-many minimal projections, say $p_1,\ldots, p_n$ in $L(\Lambda)$. Note that $\tau(p_i)=\frac{1}{n}$ for all $1\leq i\leq n$ with respect to the canonical trace $\tau$ on $L(\Gamma)$. Then since $L(\Gamma)$ is a II$_1$ factor, there exists a von Neumann subalgebra $N\subset L(\Gamma)$ such that $N\cong M_n(\mathbb{C})$ and $L(\Lambda)\subsetneq N$. This shows that $L(\Lambda)$ is not maximal amenable in $L(\Gamma)$.   
\end{remark}

\end{document}